\documentclass[12pt]{article}
\pdfoutput=1
\usepackage{e-jc}\def\volno{16}\def\volyear{2009}

\makeatletter
\renewcommand{\ps@plain}{%
\renewcommand{\@oddfoot}{\hfil\footrm\thepage}}
\def\dateline#1#2{\mbox{}}
\makeatother
\usepackage[euler-digits]{eulervm}
\usepackage{graphicx}
\usepackage{amssymb}
\usepackage{amsmath}
\usepackage[colorlinks=true]{hyperref}

\makeatletter
\DeclareRobustCommand{\textsquare}{%
  \begingroup \usefont{U}{msa}{m}{n}\thr@@\endgroup
}
\providecommand{\qedsymbol}{\textsquare}
\DeclareRobustCommand{\qed}{%
  \ifmmode \mathqed
  \else
    \leavevmode\unskip\penalty9999 \hbox{}\nobreak\hfill
    \quad\hbox{\qedsymbol}%
  \fi
}
\let\QED@stack\@empty
\let\qed@elt\relax
\newcommand{\pushQED}[1]{%
  \toks@{\qed@elt{#1}}\@temptokena\expandafter{\QED@stack}%
  \xdef\QED@stack{\the\toks@\the\@temptokena}%
}
\newcommand{\popQED}{%
  \begingroup\let\qed@elt\popQED@elt \QED@stack\relax\relax\endgroup
}
\def\popQED@elt#1#2\relax{#1\gdef\QED@stack{#2}}
\newcommand{\qedhere}{%
  \begingroup \let\mathqed\math@qedhere
    \let\qed@elt\setQED@elt \QED@stack\relax\relax \endgroup
}

\newenvironment{proof}[1][\proofname]{\par
  \pushQED{\qed}%
  \normalfont \topsep6\p@\@plus6\p@\relax
  \trivlist
  \item[\hskip\labelsep
        \itshape
    #1\@addpunct{.}]\ignorespaces
}{%
  \popQED\endtrivlist\@endpefalse
}
\providecommand{\proofname}{Proof}
\makeatother
\author{G\"otz Pfeiffer\\
Department of Mathematics\\
National University of Ireland, Galway\\
Ireland\\
\texttt{goetz.pfeiffer@nuigalway.ie}}
\title{Quiver Presentations for Descent Algebras of Exceptional Type}

\date{\dateline{Oct 15, 2008}{Feb ??, 200?}}

\let\emptyset\varnothing

\newcommand{\Q}{\mathbb{Q}}

\renewcommand{\AA}{\mathcal{A}}
\newcommand{\DD}{\mathcal{D}}
\newcommand{\PP}{\mathcal{P}}

\newcommand{\eb}{\mathbf{e}}
\newcommand{\vb}{\mathbf{v}}

\newcommand{\Size}[1]{\left|#1\right|}
\newcommand{\Norm}[1]{\left\|#1\right\|}

\newcommand{\id}{\mathsf{id}}

\newcommand{\tA}{\stackrel{.}{\to}}
\newcommand{\tB}{\stackrel{..}{\to}}
\newcommand{\tC}{\stackrel{...}{\to}}
\newcommand{\tD}{\stackrel{....}{\to}}

\newtheorem{Theorem}{Theorem}

\newtheorem{Lemma}{Lemma}

\begin{document}
\thispagestyle{empty}

\maketitle

\begin{abstract}
  The descent algebra of a finite Coxeter group $W$ is a basic
  algebra, and as such it has a presentation as quiver with relations.
  In recent work, we have developed a combinatorial framework which
  allows us to systematically compute such a quiver presentation for a
  Coxeter group of a given type.  In this article, we use that
  framework to determine quiver presentations for the descent algebras
  of the Coxeter groups of exceptional or non-crystallographic type,
  i.e., of type $E_6$, $E_7$, $E_8$, $F_4$, $H_3$, $H_4$ or $I_2(m)$.
\end{abstract}

\section{Introduction.}

Let $W$ be a finite Coxeter group.  The descent algebra  $\Sigma(W)$ is a
subalgebra of the group algebra of $W$, arising from the partition of
$W$ into descent classes \cite{Solomon1976}.
It supports an algebra homomorphism with nilpotent kernel into the (commutative)
character ring of $W$ and therefore is a basic algebra.  Each basic algebra
has a presentation as quiver with relations~\cite[Section II.3]{AssemEtAl2006}.
In recent work~\cite{pfeiffer-quiver}
we have presented an algorithm for the construction
of such a quiver presentation for any finite
Coxeter group.  
The algorithm has been implemented
in the \textsf{GAP} \cite{GAP} package \textsf{ZigZag} \cite{zigzag},
which is based on the \textsf{CHEVIE} \cite{chevie} package for
finite Coxeter groups and Iwahori--Hecke algebras.

In this article, the algorithm is applied to the Coxeter groups of
exceptional or non-crystallographic type.  The results for dihedral
groups, i.e., Coxeter groups of type $I_2(m)$, are shown in
Section~\ref{sec:i2}, the results for type $H_3$ in
Section~\ref{sec:h3}, for type $H_4$ in Section~\ref{sec:h4}, for type
$F_4$ in Section~\ref{sec:f4}, for type $E_6$ in Section~\ref{sec:e6},
for type $E_7$ in Section~\ref{sec:e7}, and for type $E_8$ in
Section~\ref{sec:e6}.  Each such section contains a labelling of the
generators of $W$ in the form of a Coxeter diagram, tabular listings
of the vertices and edges of the quiver, a picture of the
corresponding graph, a list of relations (if any), a listing of the
projective indecomposable modules and their Loewy series (verifying
results on the Loewy length by Bonnaf\'e and the
author~\cite{BonnafePfeiffer2008}), as well as the Cartan matrix.

In this section we briefly recall the setup~\cite{pfeiffer-quiver} and
provide some important general results.  For a general introduction to
Coxeter groups and related topics we refer to the
books~\cite{Bourbaki1968} and~\cite{GePf2000}.

\subsection{The descent algebra and some of its bases.}
Let $W$ be a finite Coxeter group, generated by a set $S$ of simple
reflections with corresponding length function $\ell$.  The descent
algebra $\Sigma(W)$ is defined as a subspace of the group algebra
$\Q W$ as follows.

For a subset $J \subseteq S$ we denote by $W_J$ the (parabolic)
subgroup of $W$ generated by $J$ and we define
\begin{align*}
X_J := \{x \in W : \ell(sx) > \ell(x) \text{ for all } s \in J\}
\end{align*}
as the set of minimal length right
coset representatives of $W_J$ in $W$.  Then 
\begin{align*}
  X_J^{-1} := \{x^{-1} : x \in X_J\} 
\end{align*}
is a set of left coset representatives of $W_J$ in $W$ and
it is well-known that for $J, K \subseteq S$ the intersection 
\begin{align*}
  X_{JK} := X_J \cap X_K^{-1}
\end{align*}
is a set of representatives of the $W_J$,
$W_K$-double cosets in $W$.  For $J \subseteq S$, we form the sum
\begin{align*}
  x_J := \sum X_J^{-1} \in \Q W
\end{align*}
and define $\Sigma(W)$ to be the subspace of $\Q W$ spanned by 
the elements $x_J$, $J \subseteq S$.
By Solomon's Theorem~\cite{Solomon1976}, 
$\Sigma(W)$ is in fact a subalgebra of $\Q W$, since, for $J, K \subseteq S$,
\begin{align*}
  x_J x_K = \sum_{L \subseteq S} a_{JKL} x_L
\end{align*}
where $a_{JKL} = \Size{X_{JKL}}$ and $X_{JKL} = \{d \in X_{JK} : J^d
\cap K = L\}$.

The linear independence of the $x_J$ is best seen with a different
basis.  For this, let 
\begin{align*}
  \DD(w) := \{s \in S : \ell(sw) < \ell(w)\}
\end{align*}
be the (left) descent set of $w \in W$.  Then
$X_J = \{w \in W : \DD(w) \cap J = \emptyset\}$.  We further define, for $K \subseteq S$, the descent class
\begin{align*}
  Y_K := \{w \in W : \DD(w) = S \setminus K\} = \DD^{-1}( S \setminus K),
\end{align*}
and we form the sum
\begin{align*}
  y_K := \sum Y_K^{-1} \in \Q W.
\end{align*}
Then $X_J = \bigsqcup_{K \supseteq J} Y_K$ and thus
$x_J = \sum_{K \supseteq J} y_K$.
Hence, by M\"obius inversion,
\begin{align*}
  y_K = \sum_{J \supseteq K} (-1)^{\Size{J - K}} x_J
\end{align*}
and so the $y_K$ span all of $\Sigma(W)$.
Clearly the $y_K$ are linearly independent and thus 
$\Sigma(W)$ has dimension  $2^{\Size{S}}$, the number of subsets of $S$.

The group $W$ acts on itself and on its subsets by conjugation and
thus partitions the power set $\PP(S)$ into classes of conjugate
subsets of $S$.  In this context, a third basis $\{e_L: L \subseteq S\}$ of $\Sigma(W)$ has been introduced by Bergeron,
Bergeron, Howlett and Taylor~\cite{BeBeHoTa92}, as follows.
For $J \subseteq S$, we define
\begin{align*}
  X_J^{\sharp} := \{x \in X_J : J^x \subseteq S\} 
\end{align*}
and we denote by 
\begin{align*}
  [J]:= \{J^x : x \in X_J^{\smash{\sharp}}\} \subseteq \PP(S)
\end{align*}
the class of $J$.  Moreover, we denote by 
\begin{align*}
  \Lambda := \{[J] : J \subseteq S\}
\end{align*}
the set of all classes, and for $L \in \lambda \in
\Lambda$ we set $\Norm{\lambda}:= \Size{L}$.

For $J, K \subseteq S$, we set
\begin{align*} \tag{$*$} \label{eq:star}
  m_{KL} := \sum_{J\in [K]} a_{JKL} =
  \begin{cases}
    \Size{X_K \cap X_L^{\smash{\sharp}}}, & \text{if $K \supseteq L$}, \\
    0, & \text{otherwise.}
  \end{cases}
\end{align*}
Then, for a suitable ordering of the subsets of $S$, the
matrix $(m_{KL})_{K, L \subseteq S}$ is a triangular matrix
with nonzero diagonal entries $\Size{X_{\smash{J}}^{\smash{\sharp}}} > 0$
and hence invertible.  Therefore,  the conditions
\begin{align*}
  x_K = \sum_L m_{KL} e_L,
\end{align*}
uniquely define a basis $\{e_L : L \subseteq S\}$ of $\Sigma(W)$.

Among many other properties, Bergeron, Bergeron, Howlett and Taylor
\cite{BeBeHoTa92}
show that the elements
\begin{align*}
  e_{\lambda} = \sum_{L \in \lambda} e_L, \quad \lambda \in \Lambda,
\end{align*}
form a complete set of primitive, pairwise orthogonal idempotents of
$\Sigma(W)$.

\subsection{The central case.}
The finite group $W$ has a unique element $w_0$ of maximal length
which is characterized by the property that $\DD(w_0) = S$.  We now
show how the structure of $\Sigma(W)$ is constrained if $w_0$ is
central in $W$.  With the exception of $E_6$ and $I_2(m)$, $m$ odd,
all examples of Coxeter groups considered here have a central longest
element $w_0$.  The results on the Loewy length of
$\Sigma(W)$~\cite{BonnafePfeiffer2008} have shown already that it
makes a big difference, whether $w_0$ is central or not. Here we show
that in these cases the descent algebra $\Sigma(W)$ is a direct sum of
subalgebras $\Sigma(W)^+$ and $\Sigma(W)^-$, spanned by the
idempotents $e_{\lambda}$ with $\Norm{\lambda}$ even, and the
$e_{\lambda}$ with $\Norm{\lambda}$ odd, respectively.  For $W$ of
type $B_n$, this decomposition of $\Sigma(W)$ has been observed by
Bergeron~\cite{Bergeron1992}.

For later use, we list here some properties of the 
intersection of $X_L^{\smash{\sharp}}$ and $Y_K$.
For $J \subseteq S$, we denote by $w_J$ the longest
element of the parabolic subgroup $W_J$ of $W$.
We call $u \in W$ a prefix of $w \in W$,
and write $u \leq w$, if
$\ell(u^{-1}w) = \ell(w) - \ell(u)$.
For $J \subseteq S$, the set $X_J$ can then be
described as the set of all prefixes of the longest 
coset representative $w_J  w_0$,
and $Y_J$ is the set of all $x \in X_J$ which
have $w_{S \setminus J}$ as prefix.
In other words, $Y_J$ is the
interval from $w_{S \setminus J}$ to $w_J  w_0$
in the weak Bruhat order.
\begin{Lemma}  \label{la:y-sharp}
 \renewcommand{\labelenumi}{\textup{(\roman{enumi})}}
  Let $L \subseteq K \subseteq S$.  
  \begin{enumerate}
  \item 
 $(-1)^{\Size{K}}\Size{Y_K \cap X_L^{\smash{\sharp}}} = 
\sum_{J \supseteq K} (-1)^{\Size{J}} \Size{X_J \cap X_L^{\smash{\sharp}}}$.
\item If $x \in Y_K \cap X_L^{\sharp}$ then $w_L w_{L \cup (S
    \setminus K)}$ is a prefix of $x$.
\item   In particular, $Y_L \cap X_L^{\sharp} = \{ w_L w_0 \}$.
  \end{enumerate}
\end{Lemma}

\begin{proof}
(i)  From $X_J = \bigsqcup_{K \supseteq J} Y_K$, it
follows that
\begin{align*}
  X_J \cap X_L^{\sharp} = \bigsqcup_{K \supseteq J} Y_K \cap X_L^{\sharp},
\end{align*}
and hence by M\"obius inversion that
$\Size{Y_K \cap X_L^{\smash{\sharp}}} = 
\sum_{J \supseteq K} (-1)^{\Size{J} - \Size{K}} \Size{X_J \cap X_L^{\smash{\sharp}}}$.

(ii)  Let $x \in Y_K \cap X_L^{\sharp}$.   Then $\DD(x) = S \setminus K$.
We show that if $x \in X_L^{\sharp}$ then $\DD(x)
\subseteq \DD(ux)$ for all $u \in W_L$.   It then follows, for $u = w_L$,
that $L \cup \DD(x) \subseteq \DD(w_L x)$.
For any $w \in W$, it is known~\cite[Lemma~1.5.2]{GePf2000} that 
if $J
\subseteq \DD(w)$ then
$w_J \leq w$.
Hence $w_{L \cup (S \setminus K)} \leq w_L x$
and, since $w_L$ is a common prefix of both sides,
$w_L w_{L \cup (S \setminus K)} \leq x$.

In order to show that $\DD(x) \subseteq \DD(ux)$, we argue as follows.
For $x \in X_L^{\sharp}$, it is known~\cite[Theorem 2.3.3]{GePf2000}
that if $s \in \DD(x)$ then $d:= w_L w_{L \cup \{s\}}$ is a prefix of
$x$.  Now $d \in X_L^{\sharp}$ and thus $\ell(u^d) = \ell(u)$.  It
follows that $\ell(d^{-1} u x) = \ell(u^d d^{-1} x) \leq \ell(u) +
\ell(x) - \ell(d) = \ell(ux) - \ell(d)$.  On the other hand, $\ell(ux)
- \ell(d) \leq \ell(d^{-1} ux)$.  Hence $d$ is a prefix of $ux$.
By construction, $s \in \DD(d)$. 
Thus $s \in \DD(ux)$.

(iii) If $L = K$ then $w_{L \cup (S \setminus K)} = w_0$.  And
$x \in X_L$ implies $x \leq w_L w_0$, whereas from~(ii) it follows that $w_L w_0 \leq x$ for all $x \in
Y_L \cap X_L^{\sharp}$.
\end{proof}

We can now describe the structure of the descent algebra $\Sigma(W)$
for a Coxeter group $W$ with central longest element $w_0$.

\begin{Theorem}\label{thm:central}
  Suppose $W$ is a finite Coxeter group with central longest element
  $w_0$.  If $\lambda, \mu \in \Lambda(W)$ are such that
  $\Norm{\lambda} \not\equiv \Norm{\mu} \pmod 2$ then $e_{\mu}
  \Sigma(W) e_{\lambda} = 0$.
\end{Theorem}

\begin{proof}
The longest element $w_0$ is the unique element 
with descent set $S$ and therefore, $w_0 = y_{\emptyset} \in \Sigma(W)$.
Hence, using \eqref{eq:star}, Lemma~\ref{la:y-sharp} (i) and (iii),
\begin{align*}
  w_0 &= \sum_{K \subseteq S} (-1)^{\Size{K}} x_K \\
 & = \sum_{K} (-1)^{\Size{K}} \sum_L m_{KL} e_L \\
 & = \sum_L \sum_{K \supseteq L} (-1)^{\Size{K}}  \Size{X_K \cap X_L^{\smash{\sharp}}} e_L \\
 & = \sum_L (-1)^{\Size{L}}  \Size{Y_L \cap X_L^{\smash{\sharp}}} e_L \\
 & = \sum_L(-1)^{\Size{L}}  e_L \\
 & = \sum_{\lambda \in \Lambda} (-1)^{\Norm{\lambda}} e_{\lambda}.
\end{align*}
By our hypothesis $w_0$ is central in $W$, and so it follows that
\begin{align*}
  w_0 e_{\lambda} = e_{\lambda} w_0 = (-1)^{\Norm{\lambda}} e_{\lambda}  
\end{align*}
for all $\lambda \in \Lambda$.
Hence, if $a \in e_{\mu} \Sigma(W) e_{\lambda}$
and $(-1)^{\Norm{\lambda} + \Norm{\mu}} = -1$
then 
\begin{align*}
  a = e_{\mu} a e_{\lambda}
  = e_{\mu} w_0 a w_0 e_{\lambda}
  = (-1)^{\Norm{\mu}} e_{\mu} a (-1)^{\Norm{\lambda}} e_{\lambda}
  = - e_{\mu} a e_{\lambda} = -a,
\end{align*}
and thus $a = 0$.
\end{proof}

As a consequence, the algebra $\Sigma(W)$ is the direct sum
of subalgebras $\Sigma(W)^+$ and $\Sigma(W)^-$,
generated by the orthogonal central idempotents
\begin{align*}
  \epsilon^+ & = \tfrac12 (\id + w_0) = \sum_{\Norm{\lambda} \text{ even}} e_{\lambda}
\end{align*}
and 
\begin{align*}
\epsilon^- &= \tfrac12(\id - w_0) = \sum_{\Norm{\lambda} \text{ odd}} e_{\lambda}, 
\end{align*}
respectively.  We refer to $\Sigma(W)^+$ and $\Sigma(W)^-$
as the even and the odd part of $\Sigma(W)$.
In this article, where appropriate, we present results about $\Sigma(W)$
more compactly in terms of
the even part  $\Sigma(W)^+$ and the odd part $\Sigma(W)^-$.
For example, the quiver of $\Sigma(W)$ is 
described as union of the quivers of $\Sigma(W)^+$ and $\Sigma(W)^-$.
And the Cartan matrix of $\Sigma(W)$ is described in the form 
of two smaller Cartan matrices, 
thus omitting entries which are $0$
due to Theorem~\ref{thm:central}.

\subsection{Quiver presentations.}
The descent algebra $\Sigma(W)$ supports an algebra homomorphism
$\theta$ into the character ring $R(W)$, defined for $J \subseteq S$
by
\begin{align}
  \theta(x_J) = 1_{W_J}^W,
\end{align}
the permutation character of the action of $W$ on the cosets of the
parabolic subgroup $W_J$.  Due to Solomon~\cite{Solomon1976}, the
kernel of $\theta$ coincides with the Jacobsen radical of $\Sigma(W)$.
It follows that all simple modules of $\Sigma(W)$ are $1$-dimensional
and thus that $\Sigma(W)$ is a basic algebra.  As such it has a
presentation as path algebra of a quiver with relations.  Here, we present
such quiver presentations for the finite irreducible Coxeter groups
of exceptional or non-crystallographic type.

For a general introduction to quivers in the representation
theory of finite dimensional algebras we refer to the
book~\cite{AssemEtAl2006}.
Here, a quiver is
a directed multigraph $Q = (V, E)$, consisting of a vertex set $V$ and
an edges set $E$, together with two maps
$\iota, \tau \colon E \to V$, assigning to each edge $e \in E$
a source $\iota(e) \in V$ and a target $\tau(e) \in V$.
A path of length $\ell(a) = l$ in $Q$ is a pair 
\begin{align} \label{eq:path}
  a = (v; e_1, e_2, \dots, e_l)
\end{align}
consisting of a source $v \in V$ and a sequence of $l$ edges $e_1,
e_2, \dots, e_l \in E$ such that $\iota(e_1) = v$ and $\iota(e_i) =
\tau(e_{i-1})$ for $i = 2, \dots, l$.  Let $\AA$ be the set of all
paths in $Q$.  The vertices $v \in V$ can be identified with the paths
$(v; \emptyset)$ of length $0$, and the edges $e \in E$ can be
identified with the paths $(\iota(e); e)$ of length $1$.
Concatenation of paths defines a partial multiplication on
$\AA$ as
\begin{align*}
  (v; e_1, \dots, e_l) \circ (v'; e_1', \dots, e_{l'}')
= (v; e_1, \dots, e_l, e_1', \dots, e_{l'}'),
\end{align*}
provided that $\tau(e_l) = v'$.

The path algebra $A$ of the quiver $Q$ is defined
as
\begin{align}
  A = \Q[\AA], 
\end{align}
where $a \circ a' = 0$ if the product $a \circ a'$ is not defined in $\AA$,
and otherwise multiplication is extended by linearity from $\AA$. 

In recent work~\cite{pfeiffer-quiver}, the descent algebra is constructed as
a subquotient of the quiver algebra defined by the Hasse diagram
of the power set $\PP(S)$ with respect to reverse inclusion as follows.
For $L \subseteq S$ and $s \in S$, denote
\begin{align}
  L_s = L \setminus \{s\}.
\end{align}
A path in the quiver is a sequence $L \to L_s \to (L_s)_t \to \dots$,
for some subset $L \subseteq S$ and elements $s, t, \dotsc \in L$, which we
write as a pair $(L; s, t, \dots)$. 
The length of the path $(L; s, t, \dots)$
then is $\Size{\{s, t, \dots\}}$.
 For each such path and each $r \in S$, we define
\begin{align}
  (L; s, t, \dots).r = (L^d; s^d, t^d, \dots),
\end{align}
where $d = w_L w_M$ and $M = L \cup \{r\}$.  This sets up an action of
the free monoid $S^*$ on the set of all paths.  We denote by
\begin{align}
  [L; s, t, \dots] = (L; s, t, \dots).S^*
\end{align}
the orbit of the path $(L; s, t, \dots)$ under this action.  Such an
orbit is called a street.  The subspace spanned by the streets is in
fact \cite[Theorem 6.6]{pfeiffer-quiver} a subalgebra $\Xi$ of~$A$.

For a path $(L; s, t, \dots)$ of positive length, we define
\begin{align}
  \delta(L; s, t, \dots) = (L_s; t, \dots) - (L_s; t, \dots).s
\end{align}
and
\begin{align}
  \Delta(L; s, t, \dots)  = \delta^l(L; s, t, \dots),
\end{align}
if the path $(\delta(L; s, t, \dots)$ has length $l$.  Then $\Delta$
maps $A$ into $A_0$, the subspace of paths of length $0$.  If $A_0$ is
identified with $\Sigma(W)$ via $L \mapsto e_L$ then the restriction
of $\Delta$ to $\Xi$ is a surjective anti-homomorphism from $\Xi$ to
$\Sigma(W)$ thanks to \cite[Theorem 9.5]{pfeiffer-quiver}.  Moreover,
the vertex set $\Lambda$ of the quiver of $\Sigma(W)$ is the set of
$S^*$-orbits on $\PP(S)$, the paths of length $0$, and the edges of
the quiver are images under $\Delta$ of streets $[L; s, t, \dots]$.

An algorithm which selects a suitable subset of streets as 
images of the edges of the quiver 
and expresses relations in terms of this selection
has been formulated~\cite{pfeiffer-quiver}
and implemented in the \textsf{ZigZag} package~\cite{zigzag}.
In each case this selection provides one possible way to
identify the generators of the path algebra with specific elements of
the descent algebra $\Sigma(W)$.
In the following sections,
the results of running this algorithm on particular Coxeter groups
are presented in the form of diagrams and tables.  There 
we will use streets to label vertices and
edges of quivers.  

\subsection{Notation.}  
In order to keep the descriptions of
quivers short, we use various notational conventions for dealing with
parabolic subgroups, subsets of $S$ and the edges of the quiver.

The letter $S$ is reserved for the set of Coxeter generators of $W$,
which we identify with the integers $\{1, 2, \dots, n\}$ for $n =
\Size{S}$.  Since $n < 9$ in the remaining sections, we can omit curly
braces around and commas between elements of $S$ in the tables below.
E.g, the street $[\{1,2,3,5,6\}; 1, 6]$ will be written as $[12356;
16]$.  Also, for $i \in S$, we use $S_i$ as a shorthand for $S
\setminus \{i\}$.

We denote the conjugacy class of the parabolic subgroup $W_J$ by the
isomorphism type of $W_J$.  A name $X_{jkl\ldots}$ denotes a class of
parabolic subgroups of type $X_j \times A_k \times A_l \times \dots$.
(This naming convention covers all subgroups of an irreducible finite
Coxeter group $W$, since a parabolic subgroup of $W$ is a direct
product of irreducible finite Coxeter groups with at most one factor
not of type $A$.)  In case there are several classes of parabolic
subgroups of the same type $X_{jkl}$, we use the primed names
$X_{jkl}'$, $X_{jkl}''$, \ldots, in order to distinguish them.

In a quiver, multiple edges between the same two vertices will be distinguished
by using dotted arrows $\tA$, $\tB$, \ldots

\clearpage

%%%%%%%%%%%%%%%%%%%%%%%%%%%%%%%%%%%%%%%%%%%%%%%%%%%%%%%%%%%%%%%%%%%%%%%%%%%%%
\section{Type \texorpdfstring{$I_2(m)$}{I2}.} \label{sec:i2}

The Coxeter group $W$ of type $I_2(m)$, $m \geq 3$, has Coxeter diagram:
\[
      1 \stackrel{m}- 2
\]
The structure of the descent algebra of $W$ depends on whether $m$ is
even or odd.  The longest element $w_0$ is
central in $W$  if and only if $m$ is even.  In any case, the structure of these $4$-dimensional
algebras is not particularly complicated.  For completeness, we show
the details here in two columns, on the left for $m$ even and on the
right for $m$ odd.

\begin{figure}[htbp]
  \centering
  \includegraphics[width=.4\linewidth]{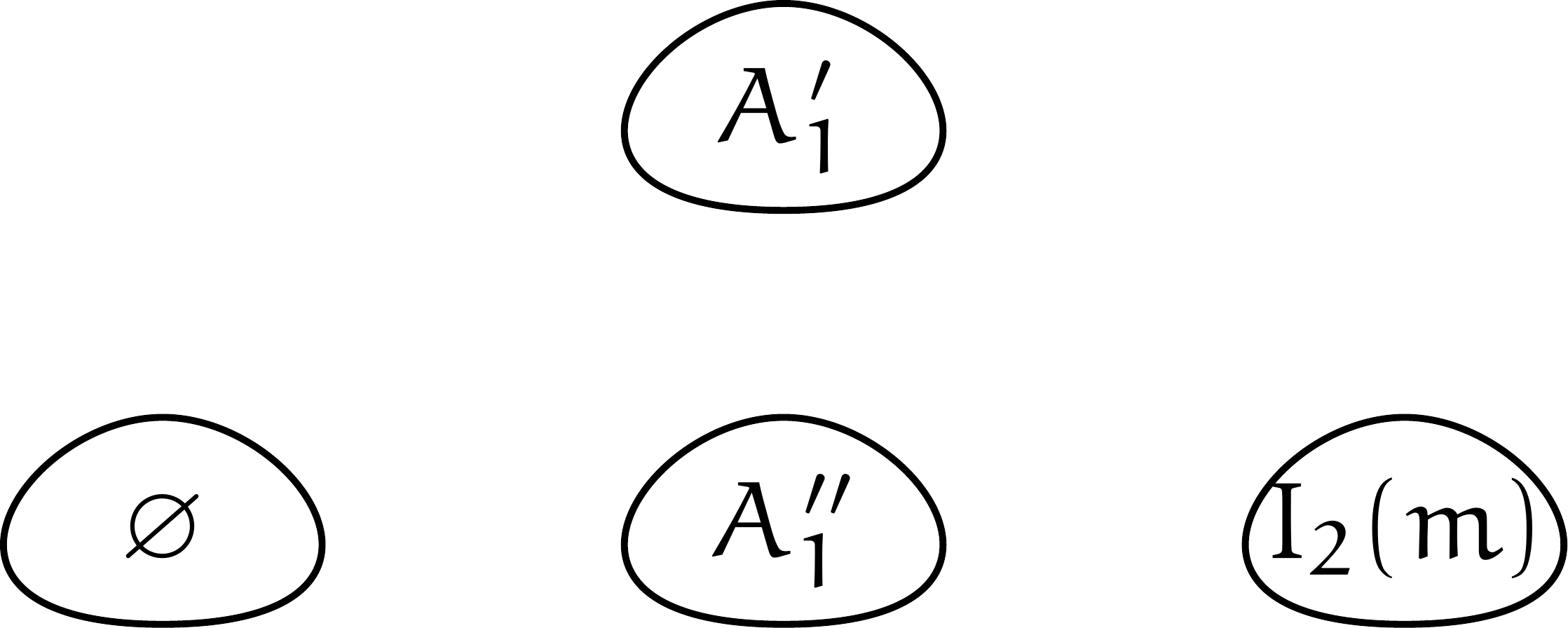}
  \qquad   \qquad  \qquad
  \includegraphics[width=.4\linewidth]{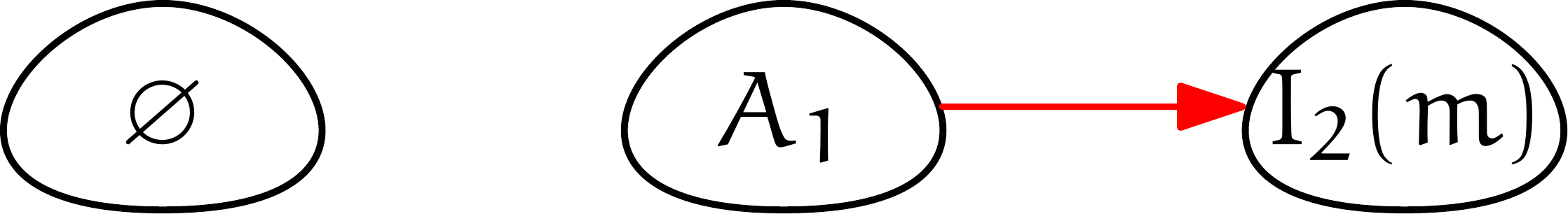}
  \caption{The quiver of type $I_2(m)$, $m$ even (left) and $m$ odd (right).}
  \label{fig:i2}
\end{figure}

\noindent
\begin{minipage}[t]{.45\linewidth}
  \paragraph{Quiver.}
  The quiver of the descent algebra $\Sigma(W)$, as shown in
  Figure~\ref{fig:i2}, has $4$ vertices and no edges, if $m$ is even.
\[
\begin{array}{|ccc|}
\hline
\vb & \textrm{type} & \lambda \\
\hline \hline
1. & \emptyset & [ \emptyset ] \\ \hline
2. & A_1' & [1] \\
3. & A_1'' & [2] \\ \hline
4. & I_2(m) & [12] \\
\hline
\end{array}
\]
\end{minipage}
\hfill
\begin{minipage}[t]{.45\linewidth}
  \paragraph{Quiver.}
  The quiver of the descent algebra $\Sigma(W)$, as shown in
  Figure~\ref{fig:i2}, has $3$ vertices and $1$ edge, if $m$ is odd.
\[
\begin{array}{|ccc|}
\hline
\vb & \textrm{type} & \lambda \\
\hline \hline
1. & \emptyset & [ \emptyset ] \\ \hline
2. & A_1 & [1] \\ \hline
3. & I_2(m) & [12] \\
\hline
\end{array}
\quad
\begin{array}{|cc|}
\hline
\eb & \alpha \\
\hline \hline
2 \to 3. & [12;1] \\
\hline
\end{array}
\]
\end{minipage}

\paragraph{Relations.}  There are no relations.  These descent
algebras are path algebras.

\medskip\noindent
\begin{minipage}[t]{.45\linewidth}
\paragraph{Projectives.}
\mbox{}\medskip

\noindent{\raggedright
$\begin{array}[b]{|c|}\hline
\emptyset\\\hline
\end{array}$,
$\begin{array}[b]{|c|}\hline
A_{1}'\\\hline
\end{array}$,
$\begin{array}[b]{|c|}\hline
A_{1}''\\\hline
\end{array}$,
$\begin{array}[b]{|c|}\hline
I_{2}(m)\\\hline
\end{array}$, if $m$ is even.
\par}
\end{minipage}
\hfill
\begin{minipage}[t]{.45\linewidth}
\paragraph{Projectives.}
\mbox{}\medskip

\noindent{\raggedright
$\begin{array}[b]{|c|}\hline
\emptyset\\\hline
\end{array}$,
$\begin{array}[b]{|c|}\hline
A_{1}\\\hline
\end{array}$,
$\begin{array}[b]{|c|}\hline
I_{2}(m)\\\hline
A_{1}\\\hline
\end{array}$,
if $m$ is odd.
\par}
\end{minipage}

\bigskip\noindent
\begin{minipage}[t]{.45\linewidth}
\paragraph{Cartan Matrix.}

\[
\begin{array}{r||r|r|}
\hline
\emptyset & 1 & .\\
\hline
I_{2}(m) & .  & 1\\
\hline
\end{array}
\qquad
\begin{array}{r||rr|}
\hline
A_{1}' & 1 & . \\
A_{1}'' & . & 1 \\
\hline
\end{array}
\]
\end{minipage}
\hfill
\begin{minipage}[t]{.45\linewidth}
\paragraph{Cartan Matrix.}

\[
\begin{array}{r||r|r|r|}
\hline
\emptyset & 1 & . & .\\
\hline
A_{1} & . & 1 & .\\
\hline
I_{2}(m) & . & 1 & 1\\
\hline
\end{array}
\]
\end{minipage}

%%%%%%%%%%%%%%%%%%%%%%%%%%%%%%%%%%%%%%%%%%%%%%%%%%%%%%%%%%%%%%%%%%%%%%%%%%%%% 
\section{Type \texorpdfstring{$H_3$}{H3}.} \label{sec:h3}

The Coxeter group $W$ of type $H_3$ has Coxeter diagram:
\[
      1 \stackrel5- 2 - 3
\]
In this group, the longest element $w_0$ is central.

\paragraph{Quiver.}
The quiver has $6$ vertices and $2$ edges.
\[
\begin{array}{|ccc|}
\hline
\vb & \textrm{type} & \lambda \\
\hline \hline
1. & \emptyset & [ \emptyset ] \\
\hline
3. & A_{11}    & [13]          \\
4. & A_2       & [23] \\ 
5. & H_2       & [12] \\ 
\hline
\end{array}
\qquad \qquad
\begin{array}{|ccc|}
\hline
\vb & \textrm{type} & \lambda \\
\hline \hline
2. & A_1 & [1] \\
\hline
6. & H_3 & [S] \\  
\hline
\end{array}
\quad
\begin{array}{|cc|}
\hline
\eb & \alpha \\
\hline \hline
2 \tA 6. & [123;12] \\
2 \tB 6. & [123;31] \\
\hline
\end{array}
\]
Figure~\ref{fig:h3} shows the quiver.
\begin{figure}[htbp]
  \centering
  \includegraphics[width=.7\linewidth]{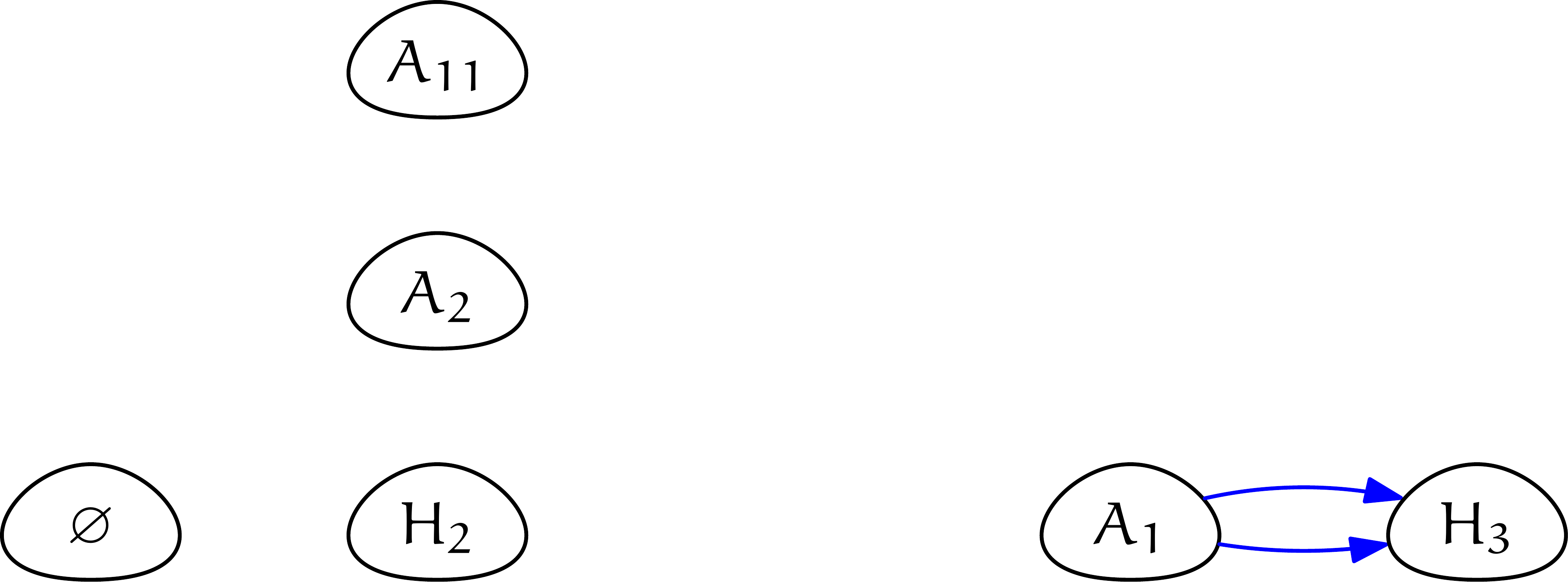}
  \caption{The quiver of type $H_3$.}
  \label{fig:h3}
\end{figure}

\paragraph{Relations.}  There are no relations.  This descent algebra
is a path algebra.

\paragraph{Projectives.}
\mbox{}

\noindent{\raggedright
$\begin{array}[b]{|c|}\hline
\emptyset\\\hline
\end{array}$,
$\begin{array}[b]{|c|}\hline
A_{1}\\\hline
\end{array}$,
$\begin{array}[b]{|c|}\hline
A_{11}\\\hline
\end{array}$,
$\begin{array}[b]{|c|}\hline
A_{2}\\\hline
\end{array}$,
$\begin{array}[b]{|c|}\hline
H_{2}\\\hline
\end{array}$,
$\begin{array}[b]{|c|}\hline
H_{3}\\\hline
(A_{1})^{2}\\\hline
\end{array}$.
\par}

\paragraph{Cartan Matrix.}

\[
\begin{array}{r||r|rrr|}
\hline
\emptyset & 1 & . & . & .\\
\hline
A_{11}    & . & 1 & . & .\\
A_{2}     & . & . & 1 & .\\
H_{2}     & . & . & . & 1\\
\hline
\end{array}
\qquad
\begin{array}{r||r|r|}
\hline
A_{1} & 1 & .\\
\hline
H_{3} & 2 & 1\\
\hline
\end{array}
\]

%%%%%%%%%%%%%%%%%%%%%%%%%%%%%%%%%%%%%%%%%%%%%%%%%%%%%%%%%%%%%%%%%%%%%%%%%%%%%
\section{Type \texorpdfstring{$H_4$}{H4}.} \label{sec:h4}

The Coxeter group $W$ of type $H_4$ has Coxeter diagram:
\[
      1 \stackrel5- 2 - 3 - 4
\]
In this group, the longest element $w_0$ is central.

\paragraph{Quiver.} The quiver, as shown in Figure~\ref{fig:h4}, has
$10$ vertices and $6$ edges.  Note that the even part is dual to the odd part.
\[
\begin{array}{|ccc|}
\hline
\vb & \textrm{type} & \lambda \\
\hline \hline
1.  & \emptyset & [\emptyset] \\ \hline
3.  & A_{11}    & [13] \\       
4.  & A_{2}     & [23] \\       
5.  & H_{2}     & [12] \\ \hline
10. & H_{4}     & [S] \\   
\hline
\end{array}
\quad
\begin{array}{|cc|}
\hline
\eb & \alpha \\
\hline \hline
3 \tA 10. & [S;23] \\
3 \tB 10. & [S;31] \\
4 \to 10. & [S;12] \\
\hline
\end{array}
\qquad \qquad
\begin{array}{|ccc|}
\hline
\vb & \textrm{type} & \lambda \\
\hline \hline
2.  & A_{1}     & [1] \\ \hline
6.  & A_{21}    & [134] \\ 
7.  & H_{21}    & [124] \\ 
8.  & A_{3}     & [234] \\ 
9.  & H_{3}     & [123] \\ 
\hline
\end{array}
\quad
\begin{array}{|cc|}
\hline
\eb & \alpha \\
\hline \hline
2 \to 8. & [234;23] \\
2 \tA 9. & [123;12] \\
2 \tB 9. & [123;31] \\
\hline
\end{array}
\]

\begin{figure}[htbp]
  \centering
  \includegraphics[width=.62\linewidth]{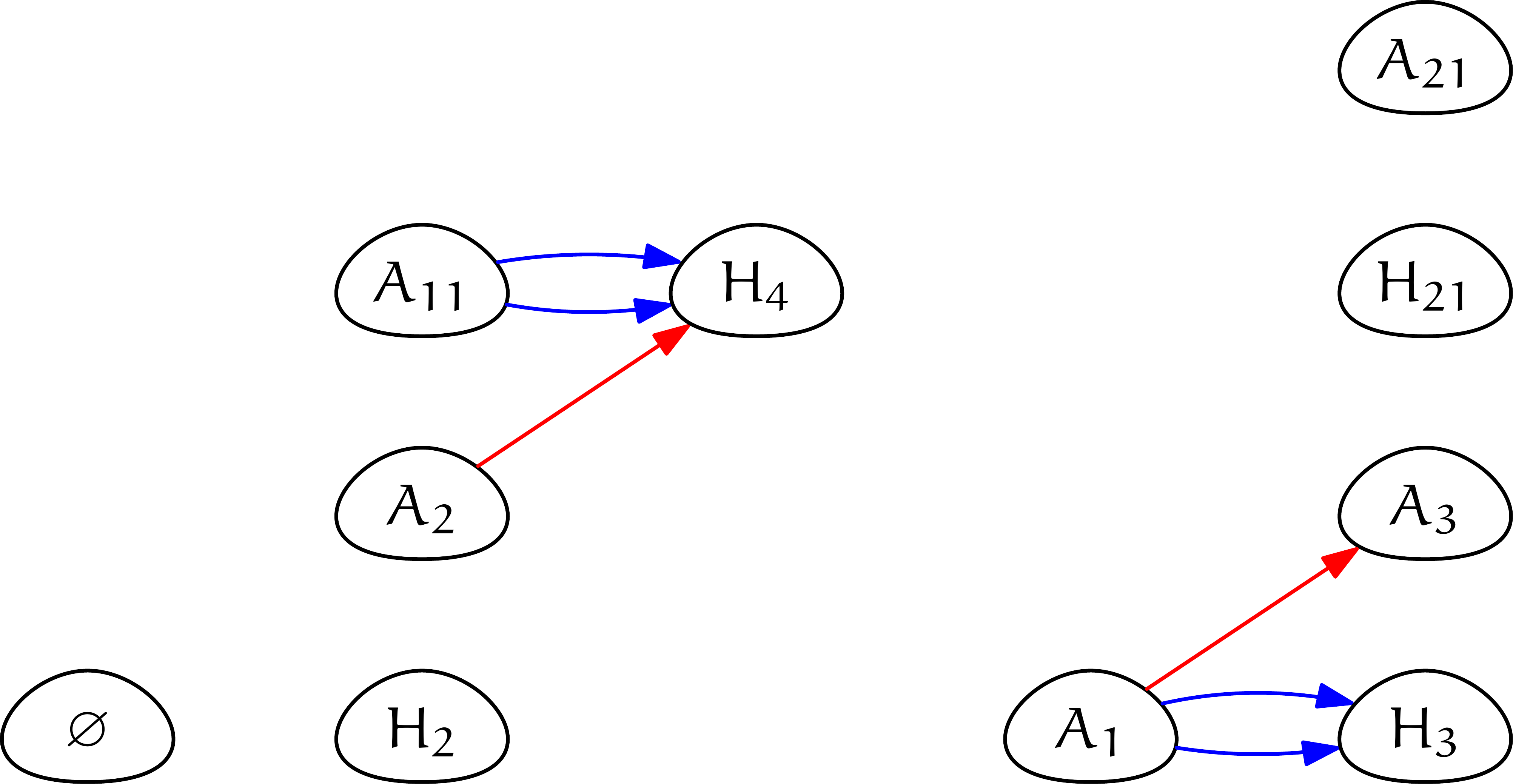}
  \caption{The quiver of type $H_4$.}
  \label{fig:h4}
\end{figure}

\paragraph{Relations.}  There are no relations.  This descent algebra
is a path algebra.

\paragraph{Projectives.}
\mbox{}

\noindent{\raggedright
$\begin{array}[b]{|c|}\hline
\emptyset\\\hline
\end{array}$,
$\begin{array}[b]{|c|}\hline
A_{1}\\\hline
\end{array}$,
$\begin{array}[b]{|c|}\hline
A_{11}\\\hline
\end{array}$,
$\begin{array}[b]{|c|}\hline
A_{2}\\\hline
\end{array}$,
$\begin{array}[b]{|c|}\hline
H_{2}\\\hline
\end{array}$,
$\begin{array}[b]{|c|}\hline
A_{21}\\\hline
\end{array}$,
$\begin{array}[b]{|c|}\hline
H_{21}\\\hline
\end{array}$,
$\begin{array}[b]{|c|}\hline
A_{3}\\\hline
A_{1}\\\hline
\end{array}$,
$\begin{array}[b]{|c|}\hline
H_{3}\\\hline
(A_{1})^{2}\\\hline
\end{array}$,
$\begin{array}[b]{|c|}\hline
H_{4}\\\hline
(A_{11})^{2}\, A_{2}\\\hline
\end{array}$.
\par}

\paragraph{Cartan Matrix.}

\[
\begin{array}{r||r|rrr|r|}
\hline
\emptyset & 1 & . & . & . & .\\
\hline
A_{11}    & . & 1 & . & . & .\\
A_{2}     & . & . & 1 & . & .\\
H_{2}     & . & . & . & 1 & .\\
\hline
H_{4}     & . & 2 & 1 & . & 1\\
\hline
\end{array}
\quad
\begin{array}{r||r|rrrr|}
\hline
A_{1}     & 1 & . & . & . & .\\
\hline
A_{21}    & . & 1 & . & . & .\\
H_{21}    & . & . & 1 & . & .\\
A_{3}     & 1 & . & . & 1 & .\\
H_{3}     & 2 & . & . & . & 1\\
\hline
\end{array}
\]

%%%%%%%%%%%%%%%%%%%%%%%%%%%%%%%%%%%%%%%%%%%%%%%%%%%%%%%%%%%%%%%%%%%%%%%%%%%%%
\section{Type \texorpdfstring{$F_4$}{F4}.} \label{sec:f4}

The Coxeter group $W$ of type $F_4$ has Coxeter diagram:
\[
      1 - 2 = 3 - 4
\]
In this group, the longest element $w_0$ is central.

\paragraph{Quiver.}
The quiver has $12$ vertices and $4$ edges.
\[
\begin{array}{|ccc|}
\hline
\vb & \textrm{type} & \lambda \\
\hline \hline
1. & \emptyset & [\emptyset] \\ \hline
4. & A_{11} & [13] \\          
5. & A_{2}' & [12] \\  
6. & A_{2}'' & [34] \\ 
7. & B_{2} & [23] \\ \hline
12. & F_{4} & [S] \\       
\hline
\end{array}
\quad
\begin{array}{|cc|}
\hline
\eb & \alpha \\
\hline \hline
4 \tA 12. & [S;23] \\
4 \tB 12. & [S;31] \\
\hline
\end{array}
\qquad\qquad
\begin{array}{|ccc|}
\hline
\vb & \textrm{type} & \lambda \\
\hline \hline
2. & A_{1}' & [1] \\
3. & A_{1}'' & [3] \\ \hline
8. & A_{21}' & [124] \\
 9. & A_{21}'' & [134] \\  
10. & B_{3}' & [123] \\    
11. & B_{3}'' & [234] \\   
\hline
\end{array}
\quad
\begin{array}{|cc|}
\hline
\eb & \alpha \\
\hline \hline
2 \to 10. & [123;31] \\
3 \to 11. & [234;23] \\
\hline
\end{array}
\]

\begin{figure}[htbp]
  \centering
  \includegraphics[width=.62\linewidth]{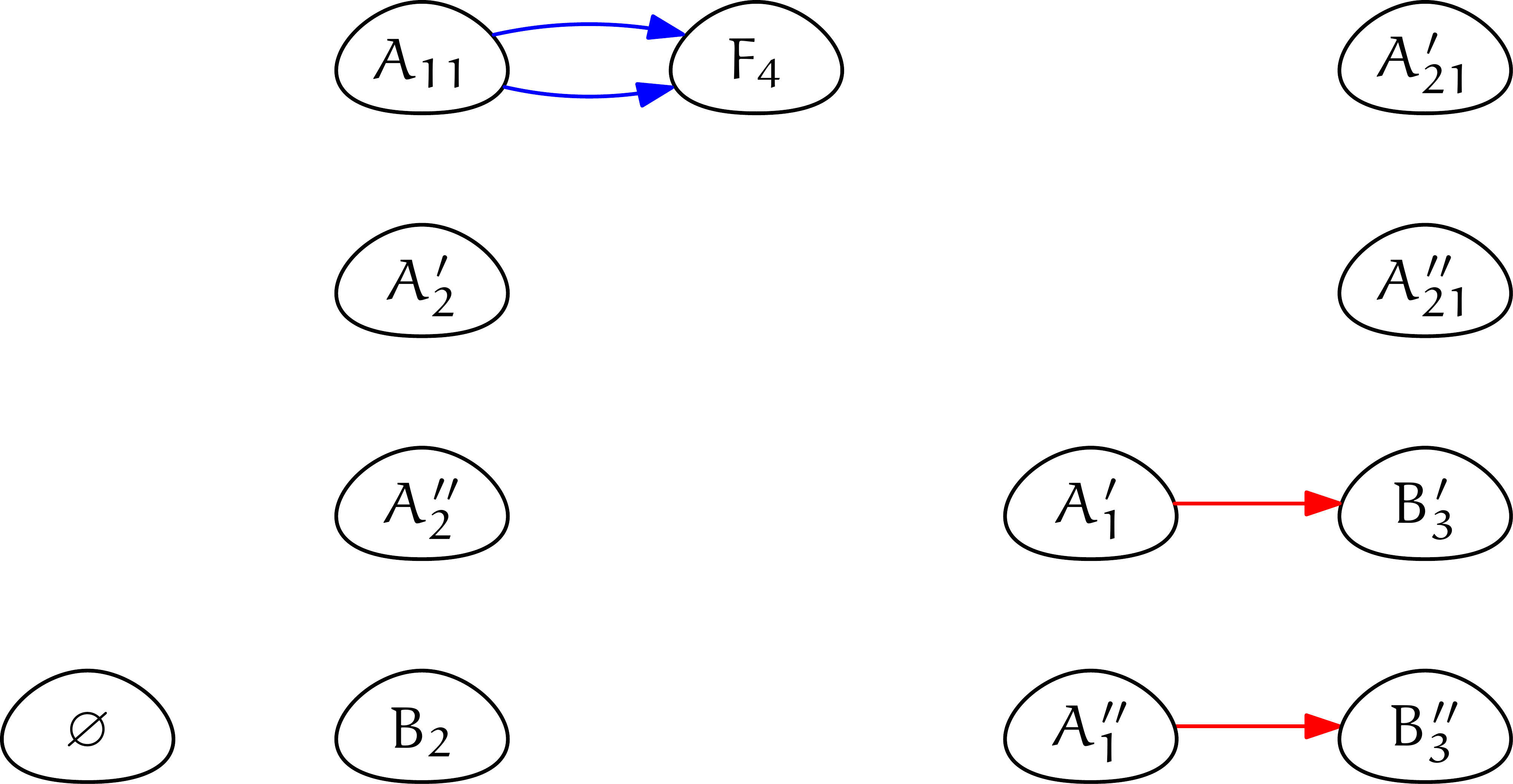}
  \caption{The quiver of type $F_4$.}
  \label{fig:f4}
\end{figure}

\paragraph{Relations.}  There are no relations.  This descent algebra
is a path algebra.

\paragraph{Projectives.}
\mbox{}

\noindent{\raggedright
$\begin{array}[b]{|c|}\hline
\emptyset\\\hline
\end{array}$,
$\begin{array}[b]{|c|}\hline
A_{1}'\\\hline
\end{array}$,
$\begin{array}[b]{|c|}\hline
A_{1}''\\\hline
\end{array}$,
$\begin{array}[b]{|c|}\hline
A_{11}\\\hline
\end{array}$,
$\begin{array}[b]{|c|}\hline
A_{2}'\\\hline
\end{array}$,
$\begin{array}[b]{|c|}\hline
A_{2}''\\\hline
\end{array}$,
$\begin{array}[b]{|c|}\hline
B_{2}\\\hline
\end{array}$,
$\begin{array}[b]{|c|}\hline
A_{21}'\\\hline
\end{array}$,
$\begin{array}[b]{|c|}\hline
A_{21}''\\\hline
\end{array}$,
$\begin{array}[b]{|c|}\hline
B_{3}'\\\hline
A_{1}'\\\hline
\end{array}$,
$\begin{array}[b]{|c|}\hline
B_{3}''\\\hline
A_{1}''\\\hline
\end{array}$,
$\begin{array}[b]{|c|}\hline
F_{4}\\\hline
(A_{11})^{2}\\\hline
\end{array}$.
\par}

\paragraph{Cartan Matrix.}

\[
\begin{array}{r||r|rrrr|r|}
\hline
\emptyset & 1 & . & . & . & . & .\\
\hline
A_{11}    & . & 1 & . & . & . & .\\
A_{2}'    & . & . & 1 & . & . & .\\
A_{2}''   & . & . & . & 1 & . & .\\
B_{2}     & . & . & . & . & 1 & .\\
\hline
F_{4}     & . & 2 & . & . & . & 1\\
\hline
\end{array}
\quad
\begin{array}{r||rr|rrrr|}
\hline
A_{1}'    & 1 & . & . & . & . & .\\
A_{1}''   & . & 1 & . & . & . & .\\
\hline
A_{21}'   & . & . & 1 & . & . & .\\
A_{21}''  & . & . & . & 1 & . & .\\
B_{3}'    & 1 & . & . & . & 1 & .\\
B_{3}''   & . & 1 & . & . & . & 1\\
\hline
\end{array}
\]

%%%%%%%%%%%%%%%%%%%%%%%%%%%%%%%%%%%%%%%%%%%%%%%%%%%%%%%%%%%%%%%%%%%%%%%%%%%%%
\section{Type \texorpdfstring{$E_6$}{E6}.} \label{sec:e6}

The Coxeter group $W$ of type $E_6$ has Coxeter diagram:
\begin{align*}\baselineskip0pt
  1 - 3 - \vbox{\hbox{$2$}\hbox{\,$|$}\hbox{$4$}} - 5 - 6
\end{align*}

\paragraph{Quiver.} The quiver, as shown in Figure~\ref{fig:e6}, has
$17$ vertices and $19$ edges.
\[
\begin{array}{|ccc|ccc|ccc|ccc|}
\hline
\vb & \textrm{type} & \lambda &
\vb & \textrm{type} & \lambda \\
\hline \hline
1. & \emptyset & [\emptyset] & 10. & A_{31} & [1245] \\  \cline{1-3}
2. & A_{1} & [1] &             11. & A_{4} & [1345] \\  \cline{1-3} 
3. & A_{11} & [12] &           12. & D_{4} & [2345] \\  \cline{4-6}
4. & A_{2} & [13] &            13. & A_{221} & [S_4] \\  \cline{1-3}
5. & A_{111} & [146] &         14. & A_{41} & [S_5] \\  
6. & A_{21} & [124] &          15. & A_{5} & [S_2] \\   
7. & A_{3} & [134] &           16. & D_{5} & [S_6] \\   \hline
8. & A_{211} & [1246] &        17. & E_{6} & [S] \\     
9. & A_{22} & [1356] &         & & \\
\hline
\end{array}
\quad
\begin{array}{|cc|cc|cc|cc|}
\hline
\eb & \alpha &
\eb & \alpha \\
\hline \hline
2 \to 7. & [134;13] &   8 \to 13. & [S_4;1] \\    
3 \to 6. & [123;1] &    8 \to 14. & [S_5;3] \\   
4 \to 11. & [1234;13] & 8 \to 17. & [S;41] \\     
5 \to 13. & [S_4;16] &  10 \to 14. & [S_5;1] \\   
5 \to 14. & [S_5;34] &  10 \to 15. & [S_2;3] \\  
5 \to 16. & [S_6;41] &  11 \to 15. & [S_2;1] \\  
6 \to 9. & [1356;1] &   11 \to 16. & [S_6;2] \\  
6 \to 10. & [1245;2] &  14 \to 17. & [S;3] \\    
6 \to 11. & [1234;3] &  16 \to 17. & [S;1] \\    
7 \to 11. & [1234;1] &  & \\                     
\hline
\end{array}
\]

\begin{figure}[htbp]
  \centering
  \includegraphics[width=.9\linewidth]{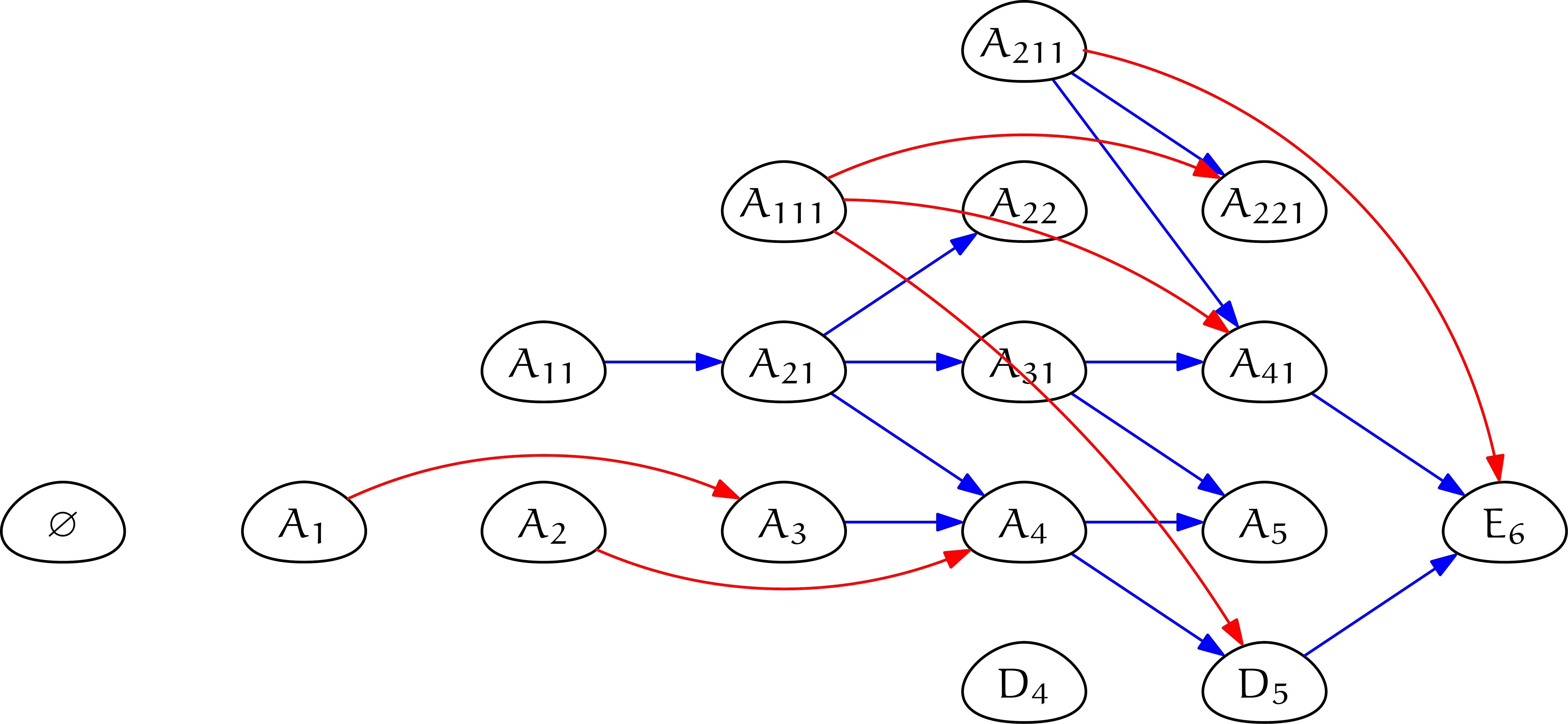}
  \caption{The quiver of type $E_6$.}
  \label{fig:e6}
\end{figure}

\paragraph{Relations.}  There are two relations, one on paths of length $2$ and one on paths of length~$3$:
\begin{align*}
(5 \to 16 \to 17) &= -2(5 \to 14 \to 17), \\
(3 \to 6 \to 11 \to 15)  &= -(3 \to 6 \to 10 \to 15). 
\end{align*}

\paragraph{Projectives.}
\mbox{}

\noindent{\raggedright
$\begin{array}[b]{|c|}\hline
\emptyset\\\hline
\end{array}$,
$\begin{array}[b]{|c|}\hline
A_{1}\\\hline
\end{array}$,
$\begin{array}[b]{|c|}\hline
A_{11}\\\hline
\end{array}$,
$\begin{array}[b]{|c|}\hline
A_{2}\\\hline
\end{array}$,
$\begin{array}[b]{|c|}\hline
A_{111}\\\hline
\end{array}$,
$\begin{array}[b]{|c|}\hline
A_{21}\\\hline
A_{11}\\\hline
\end{array}$,
$\begin{array}[b]{|c|}\hline
A_{3}\\\hline
A_{1}\\\hline
\end{array}$,
$\begin{array}[b]{|c|}\hline
A_{211}\\\hline
\end{array}$,
$\begin{array}[b]{|c|}\hline
A_{22}\\\hline
A_{21}\\\hline
A_{11}\\\hline
\end{array}$,
$\begin{array}[b]{|c|}\hline
A_{31}\\\hline
A_{21}\\\hline
A_{11}\\\hline
\end{array}$,
$\begin{array}[b]{|c|}\hline
A_{4}\\\hline
A_{2}\, A_{21}\, A_{3}\\\hline
A_{1}\, A_{11}\\\hline
\end{array}$,
$\begin{array}[b]{|c|}\hline
D_{4}\\\hline
\end{array}$,
$\begin{array}[b]{|c|}\hline
A_{221}\\\hline
A_{111}\, A_{211}\\\hline
\end{array}$,
$\begin{array}[b]{|c|}\hline
A_{41}\\\hline
A_{111}\, A_{211}\, A_{31}\\\hline
A_{21}\\\hline
A_{11}\\\hline
\end{array}$,
$\begin{array}[b]{|c|}\hline
A_{5}\\\hline
A_{31}\, A_{4}\\\hline
A_{2}\, (A_{21})^{2}\, A_{3}\\\hline
A_{1}\, A_{11}\\\hline
\end{array}$,
$\begin{array}[b]{|c|}\hline
D_{5}\\\hline
A_{111}\, A_{4}\\\hline
A_{2}\, A_{21}\, A_{3}\\\hline
A_{1}\, A_{11}\\\hline
\end{array}$,
$\begin{array}[b]{|c|}\hline
E_{6}\\\hline
A_{211}\, A_{41}\, D_{5}\\\hline
A_{111}\, A_{211}\, A_{31}\, A_{4}\\\hline
A_{2}\, (A_{21})^{2}\, A_{3}\\\hline
A_{1}\, (A_{11})^{2}\\\hline
\end{array}$.
\par}

\bigskip\noindent
Note that the projective module $E_6$ contains
copies of the simple module $A_{211}$ in the
second and the third layer of its Loewy series.
These correspond to the images under $\Delta$ of the paths
$[S;41]$ and $[S;3] \circ [S_5;3]$.

\paragraph{Cartan Matrix.}

\[
\begin{array}{r||r|r|rr|rrr|rrrrr|rrrr|r|}
\hline
\emptyset & 
1 & . & . & . & . & . & . & . & . & . & . & . & . & . & . & . & .\\
\hline
A_{1} & . & 1 & . & . & . & . & . & . & . & . & . & . & . & . & . & . & .\\
\hline
A_{11} & . & . & 1 & . & . & . & . & . & . & . & . & . & . & . & . & . & .\\
A_{2} & . & . & . & 1 & . & . & . & . & . & . & . & . & . & . & . & . & .\\
\hline
A_{111} & . & . & . & . & 1 & . & . & . & . & . & . & . & . & . & . & . & .\\
A_{21} & . & . & 1 & . & . & 1 & . & . & . & . & . & . & . & . & . & . & .\\
A_{3} & . & 1 & . & . & . & . & 1 & . & . & . & . & . & . & . & . & . & .\\
\hline
A_{211} & . & . & . & . & . & . & . & 1 & . & . & . & . & . & . & . & . & .\\
A_{22} & . & . & 1 & . & . & 1 & . & . & 1 & . & . & . & . & . & . & . & .\\
A_{31} & . & . & 1 & . & . & 1 & . & . & . & 1 & . & . & . & . & . & . & .\\
A_{4} & . & 1 & 1 & 1 & . & 1 & 1 & . & . & . & 1 & . & . & . & . & . & .\\
D_{4} & . & . & . & . & . & . & . & . & . & . & . & 1 & . & . & . & . & .\\
\hline
A_{221} & . & . & . & . & 1 & . & . & 1 & . & . & . & . & 1 & . & . & . & .\\
A_{41} & . & . & 1 & . & 1 & 1 & . & 1 & . & 1 & . & . & . & 1 & . & . & .\\
A_{5} & . & 1 & 1 & 1 & . & 2 & 1 & . & . & 1 & 1 & . & . & . & 1 & . & .\\
D_{5} & . & 1 & 1 & 1 & 1 & 1 & 1 & . & . & . & 1 & . & . & . & . & 1 & .\\
\hline
E_{6} & . & 1 & 2 & 1 & 1 & 2 & 1 & 2 & . & 1 & 1 & . & . & 1 & . & 1 & 1\\
\hline
\end{array}
\]

\clearpage

%%%%%%%%%%%%%%%%%%%%%%%%%%%%%%%%%%%%%%%%%%%%%%%%%%%%%%%%%%%%%%%%%%%%%%%%%%%%%
\section{Type \texorpdfstring{$E_7$}{E7}.} \label{sec:e7}

The Coxeter group $W$ of type $E_7$ has Coxeter diagram:
\begin{align*}\baselineskip0pt
  1 - 3 - \vbox{\hbox{$2$}\hbox{\,$|$}\hbox{$4$}} - 5 - 6 - 7
\end{align*}
In this group, the longest element $w_0$ is central.

\paragraph{Quiver.}  There are 32 vertices and 62 edges in total.
The quiver of the even part,
as shown in 
Figure~\ref{fig:e7e},
has $17$ vertices and $33$ edges.
\begin{figure}[htbp]
  \centering
  \includegraphics[width=.9\linewidth]{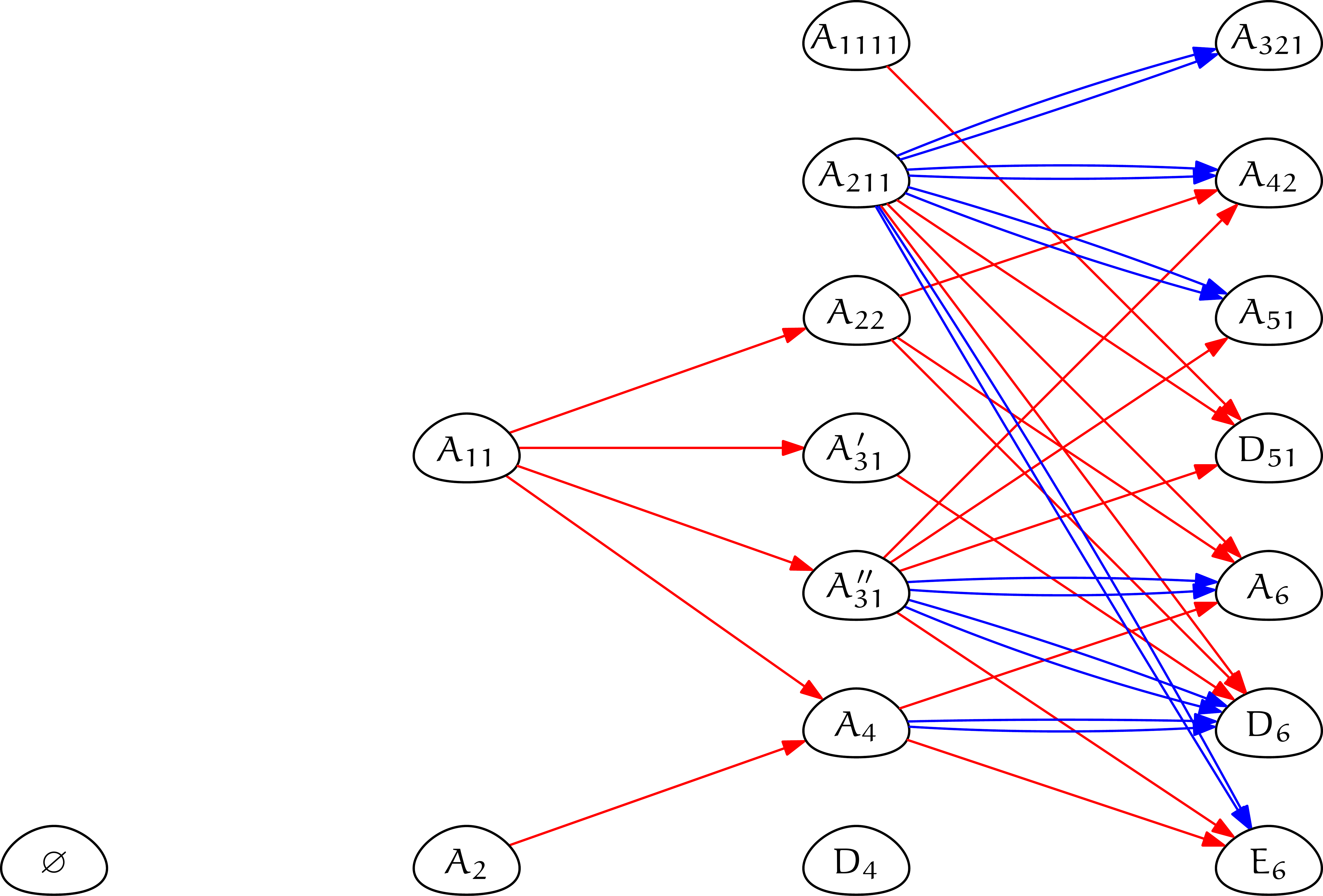}
  \caption{The even part of the quiver of type $E_7$.}
  \label{fig:e7e}
\end{figure}
\[
\begin{array}{|ccc|
ccc|ccc|ccc|
}
\hline
\vb & \textrm{type} & \lambda &
\vb & \textrm{type} & \lambda &
\vb & \textrm{type} & \lambda &
\vb & \textrm{type} & \lambda \\
\hline \hline
1. & \emptyset & [\emptyset] & 11. & A_{22} & [2467] &   25. & A_{321} & [S_4] & 30. & D_{6} & [S_1] \\
\cline{1-3}
3. & A_{11} & [12] &           12. & A_{31}' & [2457] &  26. & A_{42} & [S_5] &  31. & E_{6} & [S_7] \\
4. & A_{2} & [24] &            13. & A_{31}'' & [1245] & 27. & A_{51} & [S_3] &  & & \\
\cline{1-3}
9. & A_{1111} & [1257] &       14. & A_{4} & [1345] &    28. & D_{51} & [S_6] &  & & \\
10. & A_{211} & [1235] &       15. & D_{4} & [2345] &    29. & A_{6} & [S_2] &   & & \\
\hline
\end{array}
\]
\[
\begin{array}{|cc|
cc|cc|cc|
}
\hline
\eb & \alpha &
\eb & \alpha &
\eb & \alpha &
\eb & \alpha \\
\hline \hline
3 \to 11. & [1356;15] &  10 \tB 26. & [S_5;36] & 11 \to 29. & [S_2;45] & 13 \tB 30. & [S_1;56] \\
3 \to 12. & [2457;24] &  10 \tA 27. & [S_3;45] & 11 \to 30. & [S_1;52] & 13 \to 31. & [S_7;32] \\
3 \to 13. & [1245;24] &  10 \tB 27. & [S_3;52] & 12 \to 30. & [S_1;34] & 14 \to 29. & [S_2;13] \\
3 \to 14. & [1234;32] &  10 \to 28. & [S_6;23] & 13 \to 26. & [S_5;16] & 14 \tA 30. & [S_1;23] \\
4 \to 14. & [1234;12] &  10 \to 29. & [S_2;35] & 13 \to 27. & [S_3;24] & 14 \tB 30. & [S_1;32] \\
9 \to 28. & [S_6;41]  &  10 \to 30. & [S_1;45] & 13 \to 28. & [S_6;21] & 14 \to 31. & [S_7;12] \\
10 \tA 25. & [S_4;15] &  10 \tA 31. & [S_7;34] & 13 \tA 29. & [S_2;14] & & \\
10 \tB 25. & [S_4;56] &  10 \tB 31. & [S_7;41] & 13 \tB 29. & [S_2;41] & & \\
10 \tA 26. & [S_5;32] &  11 \to 26. & [S_5;12] & 13 \tA 30. & [S_1;24] & & \\
\hline
\end{array}
\]

The quiver of the odd part,
as shown in 
Figure~\ref{fig:e7o},
has $15$ vertices and $29$ edges.
\begin{figure}[htbp]
  \centering
  \includegraphics[width=.9\linewidth]{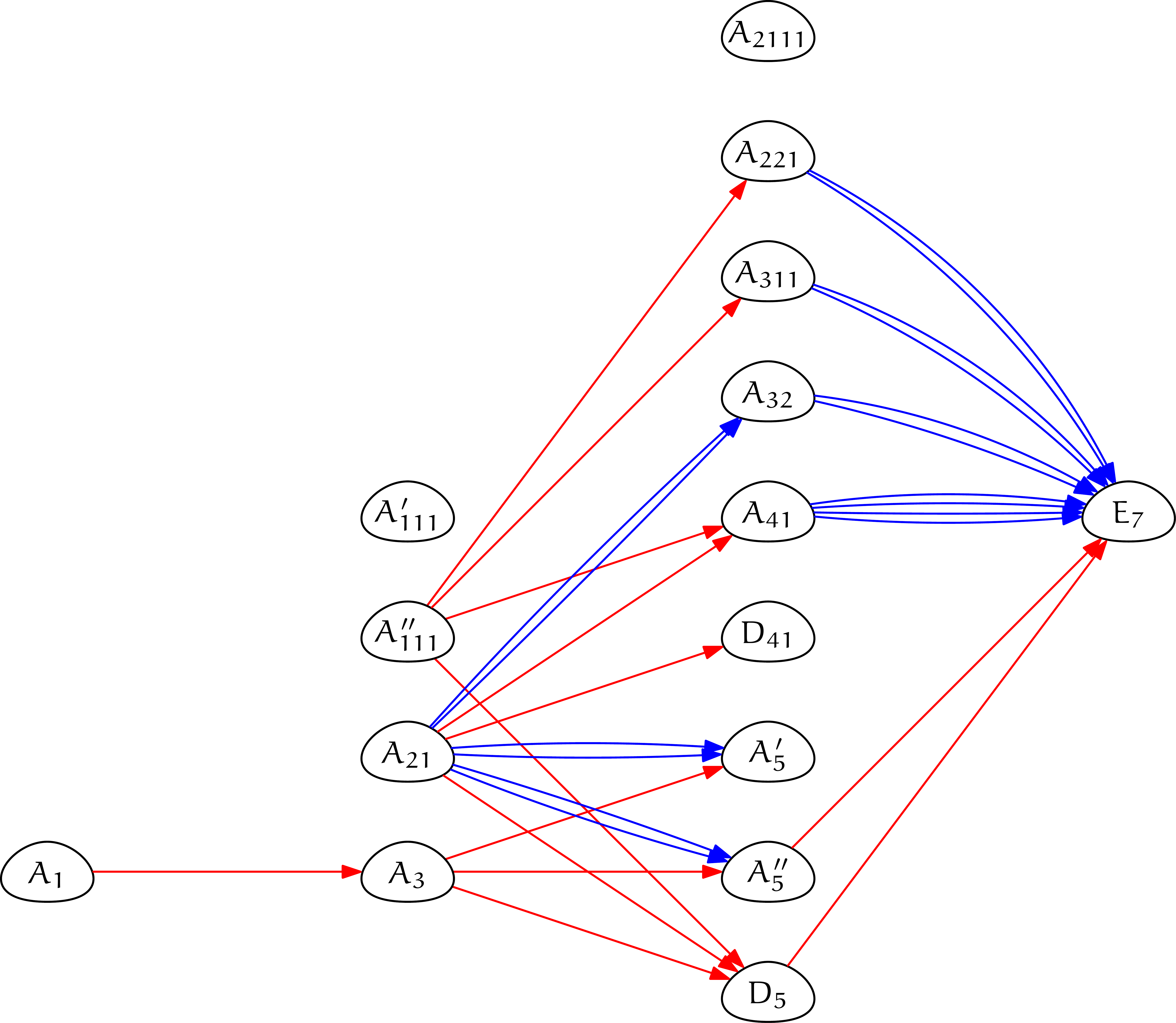}
  \caption{The odd part of the quiver of type $E_7$.}
  \label{fig:e7o}
\end{figure}
\[
\begin{array}{|ccc|
ccc|ccc|ccc|
}
\hline
\vb & \textrm{type} & \lambda &
\vb & \textrm{type} & \lambda &
\vb & \textrm{type} & \lambda &
\vb & \textrm{type} & \lambda \\
\hline \hline
2. & A_{1} & [1] &       8. & A_{3} & [234] &       19. & A_{32} & [13467] & 23. & A_{5}'' & [34567] \\ \cline{1-6}
5. & A_{111}' & [257] &  16. & A_{2111} & [12357] & 20. & A_{41} & [12347] & 24. & D_{5} & [23456] \\  \cline{10-12}
6. & A_{111}'' & [127] & 17. & A_{221} & [12367] &  21. & D_{41} & [23457] & 32. & E_{7} & [S] \\ 
7. & A_{21} & [123] &    18. & A_{311} & [23567] &  22. & A_{5}' & [24567] & & & \\
\hline
\end{array}
\]
\[
\begin{array}{|cc|
cc|cc|cc|
}
\hline
\eb & \alpha &
\eb & \alpha &
\eb & \alpha &
\eb & \alpha \\
\hline \hline
2 \to 8. & [134;13] &    7 \to 21. & [23457;23] & 8 \to 24. & [12345;21] & 20 \tB 32. & [S;32] \\ 
6 \to 17. & [12356;15] & 7 \tA 22. & [24567;45] & 17 \tA 32. & [S;45] &    20 \tC 32. & [S;56] \\ 
6 \to 18. & [12457;24] & 7 \tB 22. & [24567;52] & 17 \tB 32. & [S;53] &    20 \tD 32. & [S;62] \\ 
6 \to 20. & [12346;32] & 7 \tA 23. & [13456;34] & 18 \tA 32. & [S;34] &    23 \to 32. & [S;21] \\ 
6 \to 24. & [12345;41] & 7 \tB 23. & [13456;41] & 18 \tB 32. & [S;41] &   24 \to 32. & [S;71] \\ 
7 \tA 19. & [13467;13] & 7 \to 24. & [12345;23] & 19 \tA 32. & [S;24] &   & \\                   
7 \tB 19. & [13467;16] & 8 \to 22. & [24567;24] & 19 \tB 32. & [S;51] &   & \\                   
7 \to 20. & [12346;12] & 8 \to 23. & [13456;13] & 20 \tA 32. & [S;23] &   & \\                   
 \hline
\end{array}
\]

\paragraph{Relations.}  The presentation needs $13$ relations. There are $6$ relations on the even part:
\begin{align*}
(3 \to 13 \to 26) &= -\tfrac12(3 \to 11 \to 26), &
(3 \to 13 \tB 29) &= \tfrac12(3 \to 11 \to 29), \\
(3 \to 14 \to 29) &= -(3 \to 13 \tA 29), &
(3 \to 13 \tB 30) &= -\tfrac12(3 \to 11 \to 30), \\
(3 \to 14 \tA 30) &= -(3 \to 13 \tA 30), &
(3 \to 14 \tB 30) &= -(3 \to 12 \to 30).
\end{align*}
And there are $7$ relations on the odd part:
\begin{align*}
(6 \to 18 \tB 32) &= \tfrac12(6 \to 17 \tA 32), \\
(6 \to 20 \tB 32) &= -(6 \to 18 \tA 32), \\
(6 \to 20 \tC 32) &= -\tfrac12(6 \to 17 \tB 32), \\
(6 \to 24 \to 32) &= (6 \to 20 \tD 32) + (6 \to 20 \tC 32) - (6 \to 20 \tB 32) + (6 \to 20 \tA 32), \\
(7 \to 20 \tC 32) &= -(7 \tB 19 \tB 32), \\
(7 \tA 23 \to 32) &= (7 \tB 19 \tA 32) - (7 \to 20 \tA 32), \\
(7 \tB 23 \to 32) &= (7 \tB 19 \tA 32) - (7 \tA 19 \tA 32).
\end{align*}

\paragraph{Projectives.}
\mbox{}

\noindent
{\raggedright
$\begin{array}[b]{|c|}\hline
\emptyset\\\hline
\end{array}$,
$\begin{array}[b]{|c|}\hline
A_{1}\\\hline
\end{array}$,
$\begin{array}[b]{|c|}\hline
A_{11}\\\hline
\end{array}$,
$\begin{array}[b]{|c|}\hline
A_{2}\\\hline
\end{array}$,
$\begin{array}[b]{|c|}\hline
A_{111}'\\\hline
\end{array}$,
$\begin{array}[b]{|c|}\hline
A_{111}''\\\hline
\end{array}$,
$\begin{array}[b]{|c|}\hline
A_{21}\\\hline
\end{array}$,
$\begin{array}[b]{|c|}\hline
A_{3}\\\hline
A_{1}\\\hline
\end{array}$,
$\begin{array}[b]{|c|}\hline
A_{1111}\\\hline
\end{array}$,
$\begin{array}[b]{|c|}\hline
A_{211}\\\hline
\end{array}$,
$\begin{array}[b]{|c|}\hline
A_{22}\\\hline
A_{11}\\\hline
\end{array}$,
$\begin{array}[b]{|c|}\hline
A_{31}'\\\hline
A_{11}\\\hline
\end{array}$,
$\begin{array}[b]{|c|}\hline
A_{31}''\\\hline
A_{11}\\\hline
\end{array}$,
$\begin{array}[b]{|c|}\hline
A_{4}\\\hline
A_{11}\, A_{2}\\\hline
\end{array}$,
$\begin{array}[b]{|c|}\hline
D_{4}\\\hline
\end{array}$,
$\begin{array}[b]{|c|}\hline
A_{2111}\\\hline
\end{array}$,
$\begin{array}[b]{|c|}\hline
A_{221}\\\hline
A_{111}''\\\hline
\end{array}$,
$\begin{array}[b]{|c|}\hline
A_{311}\\\hline
A_{111}''\\\hline
\end{array}$,
$\begin{array}[b]{|c|}\hline
A_{32}\\\hline
(A_{21})^{2}\\\hline
\end{array}$,
$\begin{array}[b]{|c|}\hline
A_{41}\\\hline
A_{111}''\, A_{21}\\\hline
\end{array}$,
$\begin{array}[b]{|c|}\hline
D_{41}\\\hline
A_{21}\\\hline
\end{array}$,
$\begin{array}[b]{|c|}\hline
A_{5}'\\\hline
(A_{21})^{2}\, A_{3}\\\hline
A_{1}\\\hline
\end{array}$,
$\begin{array}[b]{|c|}\hline
A_{5}''\\\hline
(A_{21})^{2}\, A_{3}\\\hline
A_{1}\\\hline
\end{array}$,
$\begin{array}[b]{|c|}\hline
D_{5}\\\hline
A_{111}''\, A_{21}\, A_{3}\\\hline
A_{1}\\\hline
\end{array}$,
$\begin{array}[b]{|c|}\hline
A_{321}\\\hline
(A_{211})^{2}\\\hline
\end{array}$,
$\begin{array}[b]{|c|}\hline
A_{42}\\\hline
(A_{211})^{2}\, A_{22}\, A_{31}''\\\hline
A_{11}\\\hline
\end{array}$,
$\begin{array}[b]{|c|}\hline
A_{51}\\\hline
(A_{211})^{2}\, A_{31}''\\\hline
A_{11}\\\hline
\end{array}$,
$\begin{array}[b]{|c|}\hline
D_{51}\\\hline
A_{1111}\, A_{211}\, A_{31}''\\\hline
A_{11}\\\hline
\end{array}$,
$\begin{array}[b]{|c|}\hline
A_{6}\\\hline
A_{211}\, A_{22}\, (A_{31}'')^{2}\, A_{4}\\\hline
(A_{11})^{2}\, A_{2}\\\hline
\end{array}$,
$\begin{array}[b]{|c|}\hline
D_{6}\\\hline
A_{211}\, A_{22}\, A_{31}'\, (A_{31}'')^{2}\, (A_{4})^{2}\\\hline
(A_{11})^{3}\, (A_{2})^{2}\\\hline
\end{array}$,
$\begin{array}[b]{|c|}\hline
E_{6}\\\hline
(A_{211})^{2}\, A_{31}''\, A_{4}\\\hline
(A_{11})^{2}\, A_{2}\\\hline
\end{array}$,
$\begin{array}[b]{|c|}\hline
E_{7}\\\hline
(A_{221})^{2}\, (A_{311})^{2}\, (A_{32})^{2}\, (A_{41})^{4}\, A_{5}''\, D_{5}\\\hline
(A_{111}'')^{5}\, (A_{21})^{8}\, (A_{3})^{2}\\\hline
(A_{1})^{2}\\\hline
\end{array}$.
\par}

\paragraph{Cartan Matrix.}

\[
\begin{array}{r||r|rr|rrrrrrr|rrrrrrr|}
\hline
\emptyset & 1 & . & . & . & . & . & . & . & . & . & . & . & . & . & . & . & .\\
\hline
A_{11}    & . & 1 & . & . & . & . & . & . & . & . & . & . & . & . & . & . & .\\
A_{2}     & . & . & 1 & . & . & . & . & . & . & . & . & . & . & . & . & . & .\\
\hline
A_{1111}  & . & . & . & 1 & . & . & . & . & . & . & . & . & . & . & . & . & .\\
A_{211}   & . & . & . & . & 1 & . & . & . & . & . & . & . & . & . & . & . & .\\
A_{22}    & . & 1 & . & . & . & 1 & . & . & . & . & . & . & . & . & . & . & .\\
A_{31}'   & . & 1 & . & . & . & . & 1 & . & . & . & . & . & . & . & . & . & .\\
A_{31}''  & . & 1 & . & . & . & . & . & 1 & . & . & . & . & . & . & . & . & .\\
A_{4}     & . & 1 & 1 & . & . & . & . & . & 1 & . & . & . & . & . & . & . & .\\
D_{4}     & . & . & . & . & . & . & . & . & . & 1 & . & . & . & . & . & . & .\\
\hline
A_{321}   & . & . & . & . & 2 & . & . & . & . & . & 1 & . & . & . & . & . & .\\
A_{42}    & . & 1 & . & . & 2 & 1 & . & 1 & . & . & . & 1 & . & . & . & . & .\\
A_{51}    & . & 1 & . & . & 2 & . & . & 1 & . & . & . & . & 1 & . & . & . & .\\
D_{51}    & . & 1 & . & 1 & 1 & . & . & 1 & . & . & . & . & . & 1 & . & . & .\\
A_{6}     & . & 2 & 1 & . & 1 & 1 & . & 2 & 1 & . & . & . & . & . & 1 & . & .\\
D_{6}     & . & 3 & 2 & . & 1 & 1 & 1 & 2 & 2 & . & . & . & . & . & . & 1 & .\\
E_{6}     & . & 2 & 1 & . & 2 & . & . & 1 & 1 & . & . & . & . & . & . & . & 1\\
\hline
\end{array}
\]

\[
\begin{array}{r||r|rrrr|rrrrrrrrr|r|}
\hline
A_{1}     & 1 & . & . & . & . & . & . & . & . & . & . & . & . & . & .\\
\hline
A_{111}'  & . & 1 & . & . & . & . & . & . & . & . & . & . & . & . & .\\
A_{111}'' & . & . & 1 & . & . & . & . & . & . & . & . & . & . & . & .\\
A_{21}    & . & . & . & 1 & . & . & . & . & . & . & . & . & . & . & .\\
A_{3}     & 1 & . & . & . & 1 & . & . & . & . & . & . & . & . & . & .\\
\hline
A_{2111}  & . & . & . & . & . & 1 & . & . & . & . & . & . & . & . & .\\
A_{221}   & . & . & 1 & . & . & . & 1 & . & . & . & . & . & . & . & .\\
A_{311}   & . & . & 1 & . & . & . & . & 1 & . & . & . & . & . & . & .\\
A_{32}    & . & . & . & 2 & . & . & . & . & 1 & . & . & . & . & . & .\\
A_{41}    & . & . & 1 & 1 & . & . & . & . & . & 1 & . & . & . & . & .\\
D_{41}    & . & . & . & 1 & . & . & . & . & . & . & 1 & . & . & . & .\\
A_{5}'    & 1 & . & . & 2 & 1 & . & . & . & . & . & . & 1 & . & . & .\\
A_{5}''   & 1 & . & . & 2 & 1 & . & . & . & . & . & . & . & 1 & . & .\\
D_{5}     & 1 & . & 1 & 1 & 1 & . & . & . & . & . & . & . & . & 1 & .\\
\hline
E_{7}     & 2 & . & 5 & 8 & 2 & . & 2 & 2 & 2 & 4 & . & . & 1 & 1 & 1\\
\hline
\end{array}
\]

%%%%%%%%%%%%%%%%%%%%%%%%%%%%%%%%%%%%%%%%%%%%%%%%%%%%%%%%%%%%%%%%%%%%%%%%%%%%%
\section{Type \texorpdfstring{$E_8$}{E8}.} \label{sec:e8}

The Coxeter group $W$ of type $E_8$ has Coxeter diagram:
\begin{align*}\baselineskip0pt
  1 - 3 - \vbox{\hbox{$2$}\hbox{\,$|$}\hbox{$4$}} - 5 - 6 - 7 - 8
\end{align*}
In this group, the longest element $w_0$ is central.
\begin{figure}[htbp]
  \centering
  \includegraphics[width=.9\linewidth]{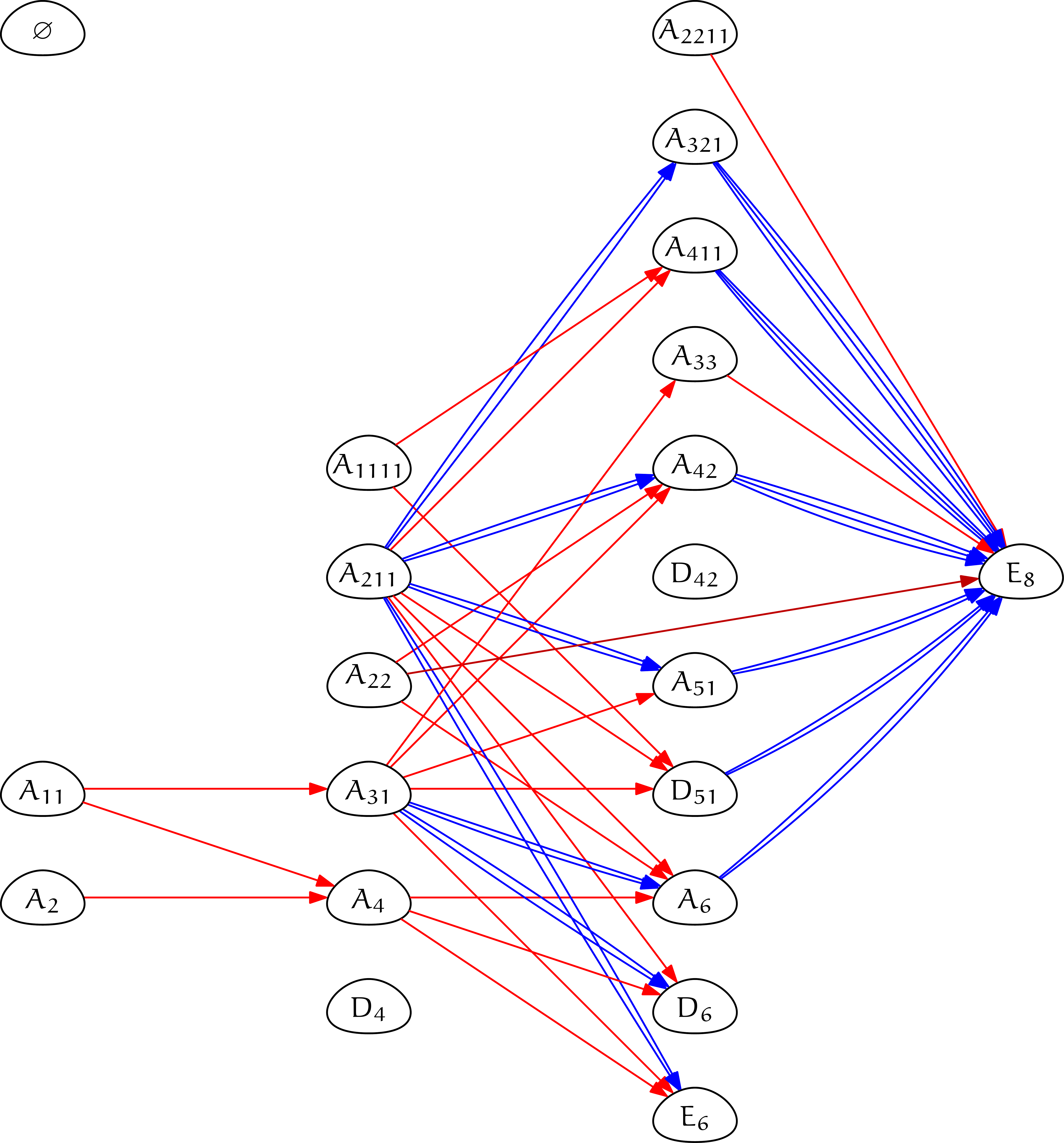}
  \caption{The even part of the quiver of type $E_8$.}
  \label{fig:e8e}
\end{figure}

\paragraph{Quiver and Relations.}  There are $41$ vertices and $109$ edges in total.  The presentation needs $33$
relations. 

The quiver of the even part,
as shown in 
Figure~\ref{fig:e8e},
has $21$ vertices and $49$ edges.
\[
\begin{array}{|ccc|
ccc|ccc|ccc|
}
\hline
\vb & \textrm{type} & \lambda &
\vb & \textrm{type} & \lambda &
\vb & \textrm{type} & \lambda &
\vb & \textrm{type} & \lambda \\
\hline
1. & \emptyset & [\emptyset] & 11. & A_{31} & [1348] &     25. & A_{33} & [134678] & 31. & D_{6} & [234567] \\  
\cline{1-3}
3. & A_{11} & [12] &           12. & A_{4} & [4567] &      26. & A_{42} & [123467] & 32. & E_{6} & [123456] \\  
\cline{10-12}
4. & A_{2} & [13] &            13. & D_{4} & [2345] &      27. & D_{42} & [234578] & 41. & E_{8} & [S] \\
\cline{1-6}
8. & A_{1111} & [1268] &       22. & A_{2211} & [123578] & 28. & A_{51} & [124567] & & & \\
9. & A_{211} & [2378] &        23. & A_{321} & [123678] &  29. & D_{51} & [123458] & & & \\
10. & A_{22} & [1367] &        24. & A_{411} & [125678] &  30. & A_{6} & [134567] &  & & \\
\hline
\end{array}
\]
\[\addtolength{\arraycolsep}{-1pt}
\begin{array}{|cc|
cc|cc|cc|
}
\hline
\eb & \alpha &
\eb & \alpha &
\eb & \alpha &
\eb & \alpha \\
\hline \hline
3 \to 11. & [1245;24] &   9 \to 30. & [134567;35] &  11 \tA 31. & [234567;24] & 25 \to 41. & [S;51] \\ 
3 \to 12. & [1234;32] &   9 \to 31. & [234567;45] &  11 \tB 31. & [234567;56] & 26 \tA 41. & [S;24] \\ 
4 \to 12. & [1234;12] &   9 \tA 32. & [123456;34] &  11 \to 32. & [123456;32] & 26 \tB 41. & [S;56] \\ 
8 \to 24. & [123468;32] & 9 \tB 32. & [123456;41] &  12 \to 30. & [134567;13] & 26 \tC 41. & [S;62] \\ 
8 \to 29. & [123457;41] & 10 \to 26. & [123467;12] & 12 \to 31. & [234567;23] & 28 \tA 41. & [S;23] \\ 
9 \tA 23. & [123567;15] & 10 \to 30. & [134567;45] & 12 \to 32. & [123456;12] & 28 \tB 41. & [S;32] \\ 
9 \tB 23. & [123567;56] & 10 \to 41. & [S;1652]    & 22 \to 41. & [S;46]      & 29 \tA 41. & [S;67] \\ 
9 \to 24. & [123468;12] & 11 \to 25. & [134678;13] & 23 \tA 41. & [S;35]      & 29 \tB 41. & [S;71] \\ 
9 \tA 26. & [123467;32] & 11 \to 26. & [123467;16] & 23 \tB 41. & [S;45] &      30 \tA 41. & [S;12] \\ 
9 \tB 26. & [123467;36] & 11 \to 28. & [124567;24] & 23 \tC 41. & [S;53] &      30 \tB 41. & [S;21] \\ 
9 \tA 28. & [124567;45] & 11 \to 29. & [123457;21] & 24 \tA 41. & [S;34] &      & \\                   
9 \tB 28. & [124567;52] & 11 \tA 30. & [134567;14] & 24 \tB 41. & [S;41] &      & \\                   
9 \to 29. & [123457;23] & 11 \tB 30. & [134567;41] & 24 \tC 41. & [S;73] &      & \\ 
\hline
\end{array}
\]

There are $16$ relations on the even part,
$14$ between paths of length $2$:
\begin{align*}
(3 \to 12 \to 30) &= -(3 \to 11 \tA 30), \\
(3 \to 12 \to 31) &= -(3 \to 11 \tA 31), \\ 
(8 \to 29 \tB 41) &= (8 \to 24 \tC 41), \\ 
(9 \to 24 \tB 41) &= -(9 \tA 23 \tB 41), \\
(9 \tA 26 \tB 41) &= -(9 \tA 23 \tC 41), \\
(9 \tB 26 \tB 41) &= (9 \tB 23 \tC 41), \\
(9 \tB 28 \tA 41) &= (9 \tA 26 \tA 41) - (9 \tB 26 \tA 41), \\
(9 \tA 28 \tB 41) &= -(9 \tA 23 \tA 41) - (9 \to 24 \tA 41), \\
(9 \tB 28 \tB 41) &= -(9 \tA 23 \tA 41) - (9 \tB 23 \tA 41), \\
(9 \to 29 \tA 41) &= -(9 \tB 26 \tC 41), \\
(9 \to 30 \tB 41) &=  -(9 \tB 26 \tA 41) + (9 \tA 28 \tA 41) + (9 \tB 28 \tA 41), \\
(10 \to 30 \tB 41) &= -(10 \to 26 \tA 41), \\
(11 \to 29 \tA 41) &= -(11 \to 26 \tC 41), \\
(11 \tB 30 \tB 41) &= (11 \to 26 \tA 41) - (11 \to 28 \tA 41) + (11 \tA 30 \tB 41), 
\end{align*}
and two between paths of length $3$ from vertex $A_{11}$ to vertex $E_8$:
\begin{align*}
(3 \to 11 \to 26 \tB 41) &= \tfrac12(3 \to 11 \to 25 \to 41), \\
(3 \to 11 \tB 30 \tB 41) &= -(3 \to 11 \to 26 \tA 41). 
\end{align*}

The quiver of the odd part,
as shown in 
Figure~\ref{fig:e8o},
has $20$ vertices and $60$ edges.
\[
\begin{array}{|ccc|
ccc|ccc|ccc|
}
\hline
\vb & \textrm{type} & \lambda &
\vb & \textrm{type} & \lambda &
\vb & \textrm{type} & \lambda &
\vb & \textrm{type} & \lambda \\
\hline \hline
2. & A_{1} & [1] &         15. & A_{221} & [12367] & 20. & A_{5} & [13456] &     36. & D_{52} & [S_6] \\
\cline{1-3}
5. & A_{111} & [147] &     16. & A_{311} & [12458] & 21. & D_{5} & [23456] &     37. & A_{7} & [S_2] \\ 
\cline{7-9}
6. & A_{21} & [124] &      17. & A_{32} & [24578] &  33. & A_{421} & [S_4] & 38. & E_{61} & [S_7] \\
7. & A_{3} & [245] &       18. & A_{41} & [24568] &  34. & A_{43} & [S_5] &  39. & D_{7} & [S_1] \\ 
\cline{1-3}
14. & A_{2111} & [12568] & 19. & D_{41} & [23458] &  35. & A_{61} & [S_3] &  40. & E_{7} & [S_8] \\ 
\hline
\end{array}
\]
\begin{figure}[htbp]
  \centering
  \includegraphics[width=.9\linewidth]{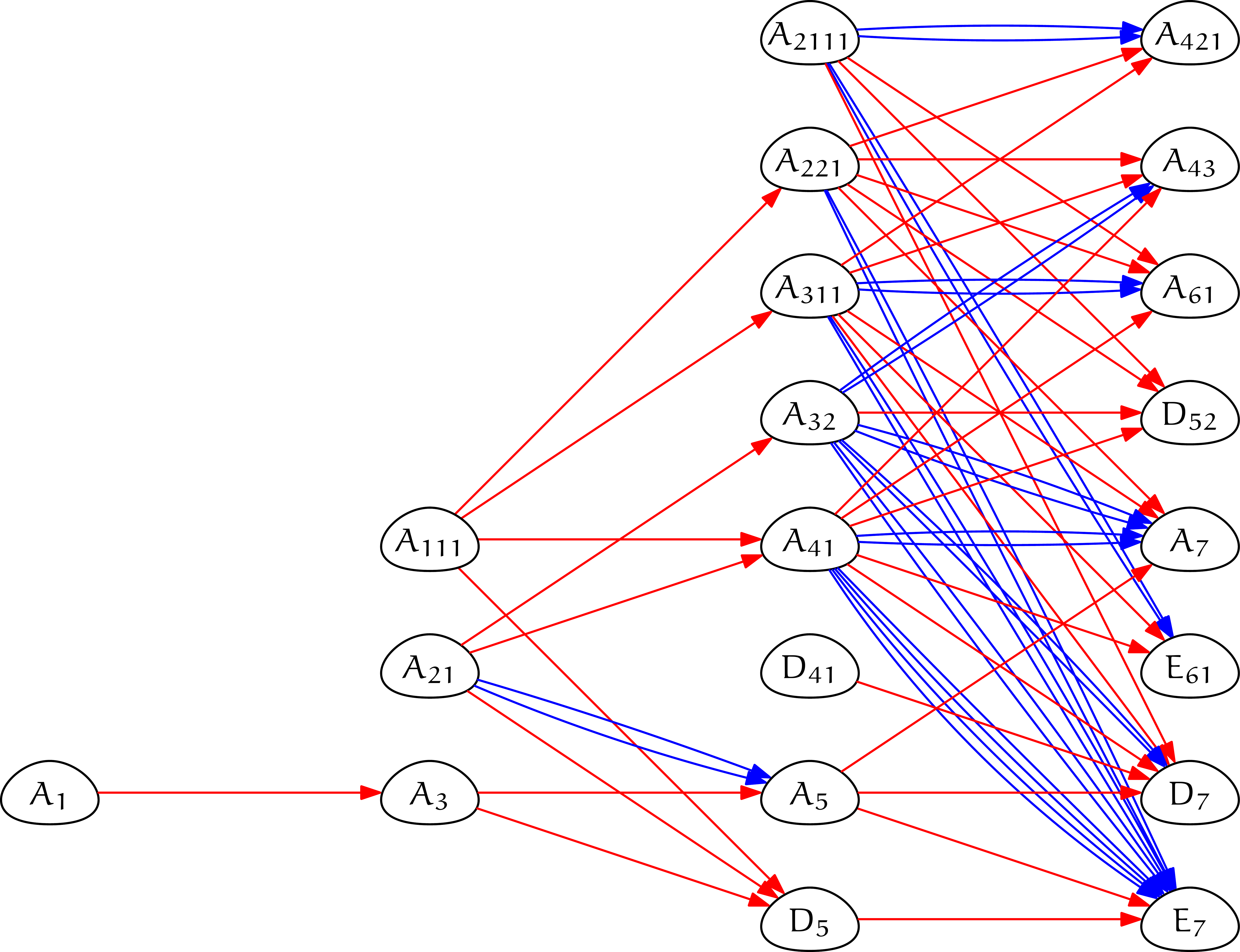}
  \caption{The odd part of the quiver of type $E_8$.}
  \label{fig:e8o}
\end{figure}
\[
\begin{array}{|cc|
cc|cc|cc|
}
\hline
\eb & \alpha &
\eb & \alpha &
\eb & \alpha &
\eb & \alpha \\
\hline \hline
2 \to 7. & [134;13] &    14 \to 36. & [S_6;41] & 16 \to 37. & [S_2;35] & 18 \to 35. & [S_3;24] \\
5 \to 15. & [12356;15] & 14 \tA 38. & [S_7;34] & 16 \to 38. & [S_7;32] & 18 \to 36. & [S_6;27] \\
5 \to 16. & [12457;24] & 14 \tB 38. & [S_7;41] & 16 \to 39. & [S_1;45] & 18 \tA 37. & [S_2;14] \\
5 \to 18. & [12346;32] & 14 \to 39. & [S_1;46] & 16 \tA 40. & [S_8;34] & 18 \tB 37. & [S_2;34] \\
5 \to 21. & [12345;41] & 15 \to 33. & [S_4;56] & 16 \tB 40. & [S_8;41] & 18 \to 38. & [S_7;12] \\
6 \to 17. & [13467;13] & 15 \to 34. & [S_5;36] & 17 \tA 34. & [S_5;12] & 18 \to 39. & [S_1;24] \\
6 \to 18. & [12346;12] & 15 \to 35. & [S_3;56] & 17 \tB 34. & [S_5;16] & 18 \tA 40. & [S_8;23] \\
6 \tA 20. & [13456;34] & 15 \to 36. & [S_6;23] & 17 \to 36. & [S_6;21] & 18 \tB 40. & [S_8;32] \\
6 \tB 20. & [13456;41] & 15 \to 37. & [S_2;46] & 17 \tA 37. & [S_2;15] & 18 \tC 40. & [S_8;56] \\
6 \to 21. & [12345;23] & 15 \tA 40. & [S_8;45] & 17 \tB 37. & [S_2;51] & 18 \tD 40. & [S_8;62] \\
7 \to 20. & [13456;13] & 15 \tB 40. & [S_8;53] & 17 \tA 39. & [S_1;25] & 19 \to 39. & [S_1;67] \\
7 \to 21. & [12345;21] & 16 \to 33. & [S_4;51] & 17 \tB 39. & [S_1;56] & 20 \to 37. & [S_2;13] \\
14 \tA 33. & [S_4;61] &  16 \to 34. & [S_5;32] & 17 \tA 40. & [S_8;24] & 20 \to 39. & [S_1;23] \\
14 \tB 33. & [S_4;67] &  16 \tA 35. & [S_3;25] & 17 \tB 40. & [S_8;51] & 20 \to 40. & [S_8;21] \\
14 \to 35. & [S_3;46] &  16 \tB 35. & [S_3;52] & 18 \to 34. & [S_5;67] & 21 \to 40. & [S_8;71] \\
\hline
\end{array}
\]

There  are $17$ relations on the odd part, all between paths of length $2$:
\begin{align*}
(5 \to 16 \to 33) &= \tfrac12(5 \to 15 \to 33), \\
(5 \to 16 \to 34) &= -\tfrac12(5 \to 15 \to 34), \\
(5 \to 18 \to 34) &= -\tfrac12(5 \to 15 \to 34), \\
(5 \to 16 \tB 35) &= -\tfrac12(5 \to 15 \to 35), \\
(5 \to 18 \to 35) &= -(5 \to 16 \tA 35), \\
(5 \to 18 \to 36) &= \tfrac12(5 \to 15 \to 36), \\
(5 \to 18 \tA 37) &= \tfrac12(5 \to 15 \to 37) - (5 \to 16 \to 37), \\
(5 \to 18 \tB 37) &= -(5 \to 16 \to 37), \\
(5 \to 16 \tB 40) &= -\tfrac12(5 \to 15 \tA 40), \\
(5 \to 18 \tC 40) &= -\tfrac12(5 \to 15 \tB 40), \\
(5 \to 18 \tB 40) &= -(5 \to 16 \to 40), \\
(5 \to 21 \to 40) &= (5 \to 18 \tD 40) + (5 \to 18 \tC 40) - (5 \to 18 \tB 40) + (5 \to 18 \tA 40), \\
(6 \to 18 \to 34) &= (6 \to 17 \tA 34), \\
(6 \tB 20 \to 37) &= -(6 \to 17 \tA 37) + (6 \to 18 \tA 37) + (6 \tA 20 \to 37), \\
(6 \tA 20 \to 37) &= -(6 \to 17 \tB 37) + (6 \to 18 \tB 37) - 2(6 \to 18 \tA 37), \\
(6 \tB 20 \to 39) &= -(6 \to 17 \tA 39) + (6 \to 18 \to 39) + (6 \tA 20 \to 39), \\
(6 \tB 20 \to 40) &= -(6 \to 17 \tA 40) + (6 \to 18 \tA 40) + (6 \tA 20 \to 40). \\
\end{align*}

\clearpage 
\paragraph{Projectives.}
\mbox{}

\noindent{\raggedright
$\begin{array}[b]{|c|}\hline
\emptyset\\\hline
\end{array}$,
$\begin{array}[b]{|c|}\hline
A_{1}\\\hline
\end{array}$,
$\begin{array}[b]{|c|}\hline
A_{11}\\\hline
\end{array}$,
$\begin{array}[b]{|c|}\hline
A_{2}\\\hline
\end{array}$,
$\begin{array}[b]{|c|}\hline
A_{111}\\\hline
\end{array}$,
$\begin{array}[b]{|c|}\hline
A_{21}\\\hline
\end{array}$,
$\begin{array}[b]{|c|}\hline
A_{3}\\\hline
A_{1}\\\hline
\end{array}$,
$\begin{array}[b]{|c|}\hline
A_{1111}\\\hline
\end{array}$,
$\begin{array}[b]{|c|}\hline
A_{211}\\\hline
\end{array}$,
$\begin{array}[b]{|c|}\hline
A_{22}\\\hline
\end{array}$,
$\begin{array}[b]{|c|}\hline
A_{31}\\\hline
A_{11}\\\hline
\end{array}$,
$\begin{array}[b]{|c|}\hline
A_{4}\\\hline
A_{11}\, A_{2}\\\hline
\end{array}$,
$\begin{array}[b]{|c|}\hline
D_{4}\\\hline
\end{array}$,
$\begin{array}[b]{|c|}\hline
A_{2111}\\\hline
\end{array}$,
$\begin{array}[b]{|c|}\hline
A_{221}\\\hline
A_{111}\\\hline
\end{array}$,
$\begin{array}[b]{|c|}\hline
A_{311}\\\hline
A_{111}\\\hline
\end{array}$,
$\begin{array}[b]{|c|}\hline
A_{32}\\\hline
A_{21}\\\hline
\end{array}$,
$\begin{array}[b]{|c|}\hline
A_{41}\\\hline
A_{111}\, A_{21}\\\hline
\end{array}$,
$\begin{array}[b]{|c|}\hline
D_{41}\\\hline
\end{array}$,
$\begin{array}[b]{|c|}\hline
A_{5}\\\hline
(A_{21})^{2}\, A_{3}\\\hline
A_{1}\\\hline
\end{array}$,
$\begin{array}[b]{|c|}\hline
D_{5}\\\hline
A_{111}\, A_{21}\, A_{3}\\\hline
A_{1}\\\hline
\end{array}$,
$\begin{array}[b]{|c|}\hline
A_{2211}\\\hline
\end{array}$,
$\begin{array}[b]{|c|}\hline
A_{321}\\\hline
(A_{211})^{2}\\\hline
\end{array}$,
$\begin{array}[b]{|c|}\hline
A_{411}\\\hline
A_{1111}\, A_{211}\\\hline
\end{array}$,
$\begin{array}[b]{|c|}\hline
A_{33}\\\hline
A_{31}\\\hline
A_{11}\\\hline
\end{array}$,
$\begin{array}[b]{|c|}\hline
A_{42}\\\hline
(A_{211})^{2}\, A_{22}\, A_{31}\\\hline
A_{11}\\\hline
\end{array}$,
$\begin{array}[b]{|c|}\hline
D_{42}\\\hline
\end{array}$,
$\begin{array}[b]{|c|}\hline
A_{51}\\\hline
(A_{211})^{2}\, A_{31}\\\hline
A_{11}\\\hline
\end{array}$,
$\begin{array}[b]{|c|}\hline
D_{51}\\\hline
A_{1111}\, A_{211}\, A_{31}\\\hline
A_{11}\\\hline
\end{array}$,
$\begin{array}[b]{|c|}\hline
A_{6}\\\hline
A_{211}\, A_{22}\, (A_{31})^{2}\, A_{4}\\\hline
(A_{11})^{2}\, A_{2}\\\hline
\end{array}$,
$\begin{array}[b]{|c|}\hline
D_{6}\\\hline
A_{211}\, (A_{31})^{2}\, A_{4}\\\hline
(A_{11})^{2}\, A_{2}\\\hline
\end{array}$,
$\begin{array}[b]{|c|}\hline
E_{6}\\\hline
(A_{211})^{2}\, A_{31}\, A_{4}\\\hline
(A_{11})^{2}\, A_{2}\\\hline
\end{array}$,
$\begin{array}[b]{|c|}\hline
A_{421}\\\hline
(A_{2111})^{2}\, A_{221}\, A_{311}\\\hline
A_{111}\\\hline
\end{array}$,
$\begin{array}[b]{|c|}\hline
A_{43}\\\hline
A_{221}\, A_{311}\, (A_{32})^{2}\, A_{41}\\\hline
A_{111}\, (A_{21})^{2}\\\hline
\end{array}$,
$\begin{array}[b]{|c|}\hline
A_{61}\\\hline
A_{2111}\, A_{221}\, (A_{311})^{2}\, A_{41}\\\hline
(A_{111})^{2}\, A_{21}\\\hline
\end{array}$,
$\begin{array}[b]{|c|}\hline
D_{52}\\\hline
A_{2111}\, A_{221}\, A_{32}\, A_{41}\\\hline
A_{111}\, (A_{21})^{2}\\\hline
\end{array}$,
$\begin{array}[b]{|c|}\hline
A_{7}\\\hline
A_{221}\, A_{311}\, (A_{32})^{2}\, (A_{41})^{2}\, A_{5}\\\hline
(A_{111})^{2}\, (A_{21})^{4}\, A_{3}\\\hline
A_{1}\\\hline
\end{array}$,
$\begin{array}[b]{|c|}\hline
E_{61}\\\hline
(A_{2111})^{2}\, A_{311}\, A_{41}\\\hline
(A_{111})^{2}\, A_{21}\\\hline
\end{array}$,
$\begin{array}[b]{|c|}\hline
D_{7}\\\hline
A_{2111}\, A_{311}\, (A_{32})^{2}\, A_{41}\, D_{41}\, A_{5}\\\hline
(A_{111})^{2}\, (A_{21})^{4}\, A_{3}\\\hline
A_{1}\\\hline
\end{array}$,
$\begin{array}[b]{|c|}\hline
E_{7}\\\hline
(A_{221})^{2}\, (A_{311})^{2}\, (A_{32})^{2}\, (A_{41})^{4}\, A_{5}\, D_{5}\\\hline
(A_{111})^{5}\, (A_{21})^{8}\, (A_{3})^{2}\\\hline
(A_{1})^{2}\\\hline
\end{array}$,
$\begin{array}[b]{|c|}\hline
E_{8}\\\hline
A_{22}\, A_{2211}\, (A_{321})^{3}\, (A_{411})^{3}\, A_{33}\, (A_{42})^{3}\, (A_{51})^{2}\, (D_{51})^{2}\, (A_{6})^{2}\\\hline
(A_{1111})^{4}\, (A_{211})^{15}\, (A_{22})^{4}\, (A_{31})^{10}\, (A_{4})^{2}\\\hline
(A_{11})^{8}\, (A_{2})^{2}\\\hline
\end{array}$.
\par}

\paragraph{Cartan Matrix.}

\[
\begin{array}{r||r|rr|rrrrrr|rrrrrrrrrrr|r|}
\hline
\emptyset & 1 & . & . & . & . & . & . & . & . & . & . & . & . & . & . & . & . & . & . & . & .\\
\hline
A_{11}    & . & 1 & . & . & . & . & . & . & . & . & . & . & . & . & . & . & . & . & . & . & .\\
A_{2}     & . & . & 1 & . & . & . & . & . & . & . & . & . & . & . & . & . & . & . & . & . & .\\
\hline
A_{1111}  & . & . & . & 1 & . & . & . & . & . & . & . & . & . & . & . & . & . & . & . & . & .\\
A_{211}   & . & . & . & . & 1 & . & . & . & . & . & . & . & . & . & . & . & . & . & . & . & .\\
A_{22}    & . & . & . & . & . & 1 & . & . & . & . & . & . & . & . & . & . & . & . & . & . & .\\
A_{31}    & . & 1 & . & . & . & . & 1 & . & . & . & . & . & . & . & . & . & . & . & . & . & .\\
A_{4}     & . & 1 & 1 & . & . & . & . & 1 & . & . & . & . & . & . & . & . & . & . & . & . & .\\
D_{4}     & . & . & . & . & . & . & . & . & 1 & . & . & . & . & . & . & . & . & . & . & . & .\\
\hline
A_{2211}  & . & . & . & . & . & . & . & . & . & 1 & . & . & . & . & . & . & . & . & . & . & .\\
A_{321}   & . & . & . & . & 2 & . & . & . & . & . & 1 & . & . & . & . & . & . & . & . & . & .\\
A_{411}   & . & . & . & 1 & 1 & . & . & . & . & . & . & 1 & . & . & . & . & . & . & . & . & .\\
A_{33}    & . & 1 & . & . & . & . & 1 & . & . & . & . & . & 1 & . & . & . & . & . & . & . & .\\
A_{42}    & . & 1 & . & . & 2 & 1 & 1 & . & . & . & . & . & . & 1 & . & . & . & . & . & . & .\\
D_{42}    & . & . & . & . & . & . & . & . & . & . & . & . & . & . & 1 & . & . & . & . & . & .\\
A_{51}    & . & 1 & . & . & 2 & . & 1 & . & . & . & . & . & . & . & . & 1 & . & . & . & . & .\\
D_{51}    & . & 1 & . & 1 & 1 & . & 1 & . & . & . & . & . & . & . & . & . & 1 & . & . & . & .\\
A_{6}     & . & 2 & 1 & . & 1 & 1 & 2 & 1 & . & . & . & . & . & . & . & . & . & 1 & . & . & .\\
D_{6}     & . & 2 & 1 & . & 1 & . & 2 & 1 & . & . & . & . & . & . & . & . & . & . & 1 & . & .\\
E_{6}     & . & 2 & 1 & . & 2 & . & 1 & 1 & . & . & . & . & . & . & . & . & . & . & . & 1 & .\\
\hline
E_{8}     & . & 8 & 2 & 4 & 15& 5 & 10& 2 & . & 1 & 3 & 3 & 1 & 3 & . & 2 & 2 & 2 & . & . & 1\\
\hline
\end{array}
\]

\[
\begin{array}{r||r|rrr|rrrrrrrr|rrrrrrrr|}
\hline
A_{1}     & 1 & . & . & . & . & . & . & . & . & . & . & . & . & . & . & . & . & . & . & .\\
\hline
A_{111}   & . & 1 & . & . & . & . & . & . & . & . & . & . & . & . & . & . & . & . & . & .\\
A_{21}    & . & . & 1 & . & . & . & . & . & . & . & . & . & . & . & . & . & . & . & . & .\\
A_{3}     & 1 & . & . & 1 & . & . & . & . & . & . & . & . & . & . & . & . & . & . & . & .\\
\hline
A_{2111}  & . & . & . & . & 1 & . & . & . & . & . & . & . & . & . & . & . & . & . & . & .\\
A_{221}   & . & 1 & . & . & . & 1 & . & . & . & . & . & . & . & . & . & . & . & . & . & .\\
A_{311}   & . & 1 & . & . & . & . & 1 & . & . & . & . & . & . & . & . & . & . & . & . & .\\
A_{32}    & . & . & 1 & . & . & . & . & 1 & . & . & . & . & . & . & . & . & . & . & . & .\\
A_{41}    & . & 1 & 1 & . & . & . & . & . & 1 & . & . & . & . & . & . & . & . & . & . & .\\
D_{41}    & . & . & . & . & . & . & . & . & . & 1 & . & . & . & . & . & . & . & . & . & .\\
A_{5}     & 1 & . & 2 & 1 & . & . & . & . & . & . & 1 & . & . & . & . & . & . & . & . & .\\
D_{5}     & 1 & 1 & 1 & 1 & . & . & . & . & . & . & . & 1 & . & . & . & . & . & . & . & .\\
\hline
A_{421}   & . & 1 & . & . & 2 & 1 & 1 & . & . & . & . & . & 1 & . & . & . & . & . & . & .\\
A_{43}    & . & 1 & 2 & . & . & 1 & 1 & 2 & 1 & . & . & . & . & 1 & . & . & . & . & . & .\\
A_{61}    & . & 2 & 1 & . & 1 & 1 & 2 & . & 1 & . & . & . & . & . & 1 & . & . & . & . & .\\
D_{52}    & . & 1 & 2 & . & 1 & 1 & . & 1 & 1 & . & . & . & . & . & . & 1 & . & . & . & .\\
A_{7}     & 1 & 2 & 4 & 1 & . & 1 & 1 & 2 & 2 & . & 1 & . & . & . & . & . & 1 & . & . & .\\
E_{61}    & . & 2 & 1 & . & 2 & . & 1 & . & 1 & . & . & . & . & . & . & . & . & 1 & . & .\\
D_{7}     & 1 & 2 & 4 & 1 & 1 & . & 1 & 2 & 1 & 1 & 1 & . & . & . & . & . & . & . & 1 & .\\
E_{7}     & 2 & 5 & 8 & 2 & . & 2 & 2 & 2 & 4 & . & 1 & 1 & . & . & . & . & . & . & . & 1\\
\hline
\end{array}
\]

%%%%%%%%%%%%%%%%%%%%%%%%%%%%%%%%%%%%%%%%%%%%%%%%%%%%%%%%%%%%%%%%%%%%%%%%%%%%%
\section{Concluding Remarks.}
\label{sec:conclude}

The quiver of the descent algebra of a Coxeter group of type $A$ has
been described by Schocker~\cite{Schocker2004}.  The quiver of the
descent algebra of a Coxeter group of type $B$ has recently been
constructed by Saliola~\cite{saliola-quiver}.  No attempts have been
made to describe the relations in these cases.  The quiver of the
descent algebra of a Coxeter group of type $D$ is not known in
general.  On the basis of the known results, and some experiments with
descent algebras of type $D$, we can classify the cases where no
relations are needed.

Suppose that $W$ is an irreducible finite Coxeter group.  Then:
\begin{itemize}
\item the descent algebra $\Sigma(W)$ is a path
  algebra only if $W$ is of type $A_n$ with $n \leq 4$, 
$B_n$ with $n \leq 5$, 
$D_n$ with $n \leq 5$,  $F_4$, $H_3$, $H_4$, or $I_2(m)$;
\item the descent algebra $\Sigma(W)$ is commutative only if $W$ is of
  type $A_1$, $B_2$, or $I_2(m)$ with $m \geq 6$ even.
\end{itemize}

%%%%%%%%%%%%%%%%%%%%%%%%%%%%%%%%%%%%%%%%%%%%%%%%%%%%%%%%%%%%%%%%%%%%%%%%%%%%%
\bibliography{descent}
\bibliographystyle{amsplain}

%%%%%%%%%%%%%%%%%%%%%%%%%%%%%%%%%%%%%%%%%%%%%%%%%%%%%%%%%%%%%%%%%%%%%%%%%%%%%
\end{document}